\documentclass[preprint,3p]{elsarticle}
\usepackage{epsfig,amsmath,amsfonts,amssymb,float}
\usepackage{subfigure} 

\numberwithin{equation}{section}
\newtheorem{theorem}{Theorem}[section]

\newtheorem{claim}[theorem]{Claim}
\newtheorem{corollary}[theorem]{Corollary}

\newtheorem{lemma}[theorem]{Lemma}

\newtheorem{remark}[theorem]{Remark}

\newenvironment{proof}[1][Proof]{\noindent\textbf{#1.} }{\ \rule{0.5em}{0.5em}}
\begin{document}
\begin{frontmatter}

\title{A Novel Discriminant Approximation of Periodic Differential Equations}

\author[rvt]{C.A.~Franco\corref{cor1}\fnref{fn1}}
\ead{cfranco@ctrl.cinvestav.mx}
\author[rvt]{J.~Collado}
\ead{jcollado@ctrl.cinvestav.mx}

\cortext[cor1]{Corresponding author}
\fntext[fn1]{The first author, acknowledges the financial support of CONACyt and CINVESTAV.}
\address[rvt]{Centro de Investigacion y de Estudios Avanzados del Instituto Politecnico Nacional, Mexico City, Mexico. Instituto Politecnico Nacional Av. No 2508, 07360}

\begin{abstract}
A new approximation of the discriminant of a second order periodic
differential equation is presented as a recursive summation of the evaluation of its excitation function at different values of time. The new
approximation is obtained, at first, by means of Walsh functions and then,
by using some algebraic properties the dependence on the Walsh functions is
eliminated. This new approximation is then used to calculate the boundaries
of stability. We prove that by letting the summation elements number approach to infinite, the discriminant approximation can be
rewritten as a summation of definite integrals. Finally we prove that the
definite integrals summation is equivalent to the discriminant approximation
made by Lyapunov which consists in an alternating series of coefficients
defined by multiple definite integrals, that is, a series of the form $%
A=A_{0}-A_{1}+\ldots +\left( -1\right) ^{n}A_{n}$, where each coefficient $%
A_{n}$ is defined as an $n-$multiple definite integral.
\end{abstract}


\end{frontmatter}
\begin{center}
\section*{Notation}
\end{center}

{\footnotesize 
\begin{tabular}{lllll}
\multicolumn{2}{l}{$\ddot{x}+\left( \alpha +\beta p\left( t\right) \right)
x=0$} & \multicolumn{2}{l}{Hill Equation} &  \\ 
$t$ & Independent Variable &  & $e_{2^{k}}$ & $=\left[ 
\begin{array}{cccc}
0 & 0 & \ldots  & 1%
\end{array}%
\right] ^{T}$ \\ 
$p\left( t\right) $ & $=p\left( t+\tau \right) ,$ Excitation Function &  & $e
$ & $=\left[ 
\begin{array}{cccc}
1 & 1 & \ldots  & 1%
\end{array}%
\right] ^{T}$ \\ 
$\tau $ & Differential Equation Period &  & $\Lambda _{\gamma }$ & $p\left(
t\right) $ Sampling Matrix \\ 
$\phi _{i}$ & Characteristic Multipliers &  & $\Gamma $, $\Gamma _{P}$ & 
Discriminant Sampling Matrices \\ 
$\Delta \left( \alpha ,\beta \right) $ or $\Delta \left( \lambda \right) $ & 
Discriminant &  & $\bar{\Lambda}_{\bar{r}}$ & $=W_{H}\Lambda _{\bar{r}}W_{H}$
\\ 
$A$ & Lyapunov Characteristic Constant &  & $\bar{P}$ & $=W_{H}PW_{H}$ \\ 
$\Phi \left( t,t_{0}\right) $ & State Transition Matrix &  & $\bar{P}^{-1}$
& $=W_{H}\left( P^{-1}\right) W_{H}$ \\ 
$M$ & Monodromy Matrix ($\Phi \left( \tau ,0\right) $) &  & $\bar{\Gamma}$ & 
$=W_{H}^{-1}\Gamma W_{H}$ \\ 
$\lambda _{i}$, $\lambda _{i}^{\prime }$ & $\left\vert \Delta \left( \lambda
\right) \right\vert =2$, $\left\vert \Delta \left( \lambda \right)
\right\vert =-2$ &  & $\bar{\Gamma}_{P}$ & $=W_{H}^{-1}\Gamma _{P}W_{H}$ \\ 
$w_{n}\left( t\right) $ & $n-th$ Walsh Function &  & $b_{n}$ & $2^{k}-n$
entry of the last column of $\bar{\Gamma}$ \\ 
$\bar{w}_{n}\left( t\right) $ & Vector of Walsh Functions $\in 
\mathbb{R}
^{n\times 1}$ &  & $c_{n}$ & $2^{k}-n$ entry of the last column of $\bar{%
\Gamma}_{P}$ \\ 
$W_{H}$ & Walsh Matrix &  & $S_{h}$ & $=\sum_{i=0}^{h}b_{i}$ \\ 
$P$ & Walsh Function Integration Operator &  & $Z_{h}$ & $%
=\sum_{i=0}^{h}c_{i}$ \\ 
$\nu _{n}$, $r_{n}$, \ldots  & Real Constants &  & $\psi _{h}$ & $=\frac{%
4\tau ^{2}}{2^{2k+2}+\tau ^{2}\left( \alpha +\beta p_{2^{k}-h}\right) }$ \\ 
$\bar{\nu}$, $\bar{r}$, \ldots  & Real Constant Vectors &  & $\xi _{h}$ & $=\alpha +\beta p_{2^{k}-h}$ \\ 
$I_{n}$ & Identity Matrix &  & $\mu _{h}$ & $=\alpha +\frac{\beta }{2}\left(
p_{2^{k}-h}+p_{2^{k}-h+1}\right) $ \\ 
$\Lambda _{i}^{\left( m\right) }$ & Permutation Matrix &  & $\delta $ & $=%
\frac{\tau }{2^{k}}$ \\ 
$p_{n}$ & $=p\left( n\frac{\tau }{2^{k}}\right) $ &  & $q\left( t\right) $ & 
$=\alpha +\beta p\left( t\right) $ \\ 
$e_{1}$ & $=\left[ 
\begin{array}{cccc}
1 & 0 & \ldots  & 0%
\end{array}%
\right] ^{T}$ &  &  & 
\end{tabular}
}

\section{Introduction}

Linear periodic differential equations can describe the dynamical behaviour
of a large number of mechanical systems. They arise quite frequently as the
result of linearising a non-linear system about a periodic solution. An
important second order example is the Hill equation%
\begin{equation}
\ddot{x}+\left( \alpha +\beta p\left( t\right) \right) x=0\text{, }p\left(
t+\tau \right) =p\left( t\right)  \label{Hi}
\end{equation}%
where $\tau $ is the minimum period of the excitation function $p\left(
t\right) $, and $p\left( t\right) $ is assumed piecewise continuous.
Equation (\ref{Hi}) has been used to describe problems in engineering and
physics, including problems in mechanics, astronomy and the theory of
electric circuits \cite{Richards,dvorak}. As is well known the
analytic solution of (\ref{Hi}) can not been obtained and the analysis of
stability hinge on the numerical calculation of its solutions.

We can prove that the stability of the solutions of the periodic
differential equation (\ref{Hi}) is defined by an autonomous function which
depends only on its constant parameters $\alpha $ and $\beta $, and on the
minimum period of its excitation function $p\left( t\right) $. This
remarkable function is known as the discriminant associated to the Hill
equation (\ref{Hi}) and it is denoted as $\Delta \left( \alpha ,\beta
\right) $. In \cite{Lyapunov} Lyapunov obtained an approximation of the
discriminant in terms of the alternating series\footnote{%
The discriminant was denoted by Lyapunov as the letter $A$. Most of the
literature uses the $\Delta \left( \alpha ,\beta \right) $ notation. The
main difference between both of them is that the former is one half of the
latter, that is $A=\frac{1}{2}\Delta \left( \alpha ,\beta \right) $.}%
\begin{equation}
A=A_{0}-A_{1}+A_{2}-A_{3}+\ldots +\left( -1\right) ^{n}A_{n}+\ldots
\label{sss}
\end{equation}%
the coefficients $A_{k}$ are defined, applying successive approximation \cite%
{Luenb}, as%
\begin{eqnarray*}
A_{0} &=&1 \\
A_{1} &=&\frac{\tau }{2}\int_{0}^{\tau }q\left( t\right) dt \\
A_{2} &=&\frac{1}{2}\int_{0}^{\tau }\int_{0}^{t_{1}}\left( \tau
-t_{1}+t_{2}\right) \left( t_{1}-t_{2}\right) q\left( t_{1}\right) q\left(
t_{2}\right) dt_{2}dt_{1} \\
&&\vdots \\
A_{n} &=&\frac{1}{2}\int_{0}^{\tau }\int_{0}^{t_{1}}\ldots
\int_{0}^{t_{n-1}}\left( \tau -t_{1}+t_{n}\right) \left( t_{1}-t_{2}\right)
\ldots \left( t_{n-1}-t_{n}\right) q\left( t_{1}\right) \ldots q\left(
t_{n}\right) dt_{1}dt_{2}\ldots dt_{n}
\end{eqnarray*}%
where $q\left( t\right) =\alpha +\beta p\left( t\right) $ and $\tau $ is the
minimum period of $p\left( t\right) $, i.e. $q\left( t+\tau \right) =q\left(
t\right) $. In \cite{LyapunovApp} Lyapunov did a detailed analysis of his
approximation. He obtained the series (\ref{sss}) by considering the equation 
\[
\ddot{x}+\lambda q\left( t\right) x=0 
\]%
instead of (\ref{Hi}), and then expanding the characteristic constant $%
A\left( \lambda \right) $ in powers of $\lambda $, i.e.\ $A\left( \lambda
\right) =A_{0}-\lambda A_{1}+\lambda ^{2}A_{2}\ldots +\left( -\lambda
\right) ^{n}A_{n}+\ldots $, and then setting $\lambda =1$ for obtaining (\ref%
{sss}). Some other discriminant approximation may be found in \cite%
{hochstadt2} and in \cite{shi}, the former approximation is based on some
properties of the Sturm-Liouville problem and the latter is based on the
successive approximation method.

For clarifying the notion of discriminant $\Delta \left( \alpha ,\beta
\right) $, let $x_{1}$ and $x_{2}$ be two solutions of (\ref{Hi}) subject to
the initial conditions%
\begin{eqnarray*}
x_{1}\left( 0\right) &=&1\text{, }x_{2}\left( 0\right) =0 \\
\dot{x}_{1}\left( 0\right) &=&0\text{, }\dot{x}_{2}\left( 0\right) =1
\end{eqnarray*}%
It is known \cite{MagnWin} that if $\rho $ is solution of the characteristic
equation 
\[
\rho ^{2}-\left( x_{1}\left( \tau \right) +\dot{x}_{2}\left( \tau \right)
\right) \rho +1=0 
\]%
then (\ref{Hi}) has at least one solution of the form%
\begin{equation}
x\left( \tau +t\right) =\rho x\left( t\right)  \label{s}
\end{equation}%
the constant $\Delta \left( \alpha ,\beta \right) =x_{1}\left( \tau \right) +%
\dot{x}_{2}\left( \tau \right) $ is known as the discriminant of (\ref{Hi}).

The roots of the characteristic equation are%
\[
\rho _{1,2}=\frac{\Delta \left( \alpha ,\beta \right) \pm \sqrt{\Delta
\left( \alpha ,\beta \right) ^{2}-4}}{2} 
\]%
from where one can notice that: If $-2<\Delta \left( \alpha ,\beta \right)
<2 $, the roots $\rho _{1,2}$ are complex conjugated numbers and lie on the
unitary circle, thus $x_{1}\left( t\right) $ and $x_{2}\left( t\right) $ are
bounded so all the solutions of (\ref{Hi}) are bounded: If $\Delta \left(
\alpha ,\beta \right) =\pm 2$, $\rho =\pm 1$, there is at least one $\tau $%
-periodic or $\tau $-anti periodic solution\footnote{%
A solution $x\left( t\right) $ is $\tau $-anti periodic if $x\left( t+\tau
\right) =-x\left( t\right) $.} and if the modulo of $\Delta \left( \alpha
,\beta \right) $ is greater than $2$ then, there is one bounded and one
unstable solution. So the stability of the solutions of (\ref{Hi}) are
determined by its discriminant $\Delta \left( \alpha ,\beta \right) $, for
further information see section 2.

The aim of the present work is: a) to introduce a new approximation for the
discriminant $\Delta \left( \alpha ,\beta \right) $. This is accomplished by
doing a series of assumptions on a new approximation obtained by means of
Wash series and then reducing it to a recursive summation, in terms of
evaluation of the excitation function at different values of time $t_{n}\in %
\left[ 0,\tau \right] $, $n=1,2,\ldots $; and, b) based on the new form of $%
\Delta \left( \alpha ,\beta \right) $, we give an alternative proof to the
discriminant approximation obtained by Lyapunov in \cite{Lyapunov} and \cite%
{LyapunovApp}. The new approximation may be seen as a "discrete" form of the
Lyapunov approximation.

This work is structured as follows: In section 2 we give a brief
introduction to the theory of periodic differential equations and Walsh series;
Section 3 and 4 are dedicated to give a rough approximation of the
discriminant of (\ref{Hi}); in section 5 we give the new approximation of $%
\Delta \left( \alpha ,\beta \right) $ and the alternative proof of the
Lyapunov approximation is done; and in section 6 we obtain the transition curves of a Hill equation.

\section{Preliminaries}

\subsection{Background for 2nd order linear periodic differential equations}

Consider the linear periodic equation $\ddot{x}+\left( \alpha +\beta p\left(
t\right) \right) x=0$ where $p\left( \tau +t\right) =p\left( t\right) $, by
the usual change of variable $z_{1}=x$ and $z_{2}=\dot{x}$ the differential
equation can be rewritten as 
\begin{equation}
\left[ 
\begin{array}{c}
\dot{z}_{1} \\ 
\dot{z}_{2}%
\end{array}%
\right] =\left[ 
\begin{array}{cc}
0 & 1 \\ 
-\left( \alpha +\beta p\left( t\right) \right) & 0%
\end{array}%
\right] \left[ 
\begin{array}{c}
z_{1} \\ 
z_{2}%
\end{array}%
\right]  \label{zzz}
\end{equation}

It is well known that the general solution of (\ref{zzz}) is the state
transition matrix $\Phi \left( t,t_{0}\right) $ whose columns are linearly
independent solutions of (\ref{zzz}) subject to the initial conditions $\Phi
\left( t_{0},t_{0}\right) =I_{2}$, in other words, let $x_{1}$ and $x_{2}$
be two linearly solutions of the periodic equation $\ddot{x}+\left( \alpha
+\beta p\left( t\right) \right) x=0$ subject to the initial conditions 
\[
\begin{array}{cc}
x_{1}\left( t_{0}\right) =1 & x_{2}\left( t_{0}\right) =0 \\ 
\dot{x}_{1}\left( t_{0}\right) =0 & \dot{x}_{2}\left( t_{0}\right) =1%
\end{array}%
\]%
then the state transition matrix is 
\[
\Phi \left( t,t_{0}\right) =\left[ 
\begin{array}{cc}
x_{1}\left( t\right) & x_{2}\left( t\right) \\ 
\dot{x}_{1}\left( t\right) & \dot{x}_{2}\left( t\right)%
\end{array}%
\right] \text{, \ \ \ \ \ \ \ \ \ }t\in \left( -\infty ,\infty \right) 
\]%
The importance of the matrix $\Phi \left( t,t_{0}\right) $ lies in the fact
that all solutions of any differential equation, in particular our equation (%
\ref{zzz}), can be expressed in terms of the transition matrix as $z\left(
t\right) =\Phi \left( t,t_{0}\right) z\left( t_{0}\right) $, for all $%
z\left( t_{0}\right) $.

The matrix $\Phi \left( t,t_{0}\right) $ of a periodic system may be written
as a multiplication of three matrices, two of them are time dependent
matrices and the other is a constant matrix; it can be proved that one of
the time dependent matrices is bounded and periodic, and the other is an
exponential one: the former gives us information about the phase of the
solutions; and the latter contains information about the growth of the
solutions, see \cite{Gelfandlidski,Yak_Str}. The next Theorem, due
to Floquet, gives us the mentioned factorization

\begin{theorem} \label{theorem1}
The state transition matrix $\Phi \left( t,t_{0}\right) \in 
\mathbb{R}
^{2\times 2}$ associated to the periodic differential equation (\ref{zzz})
has the form%
\[
\Phi \left( t,t_{0}\right) =P^{-1}\left( t\right) e^{B\left( t-t_{0}\right)
}P\left( t_{0}\right) 
\]%
where $P\left( t\right) $ and $B$ are $2\times 2$ matrices, $P\left( t+\tau
\right) =P\left( t\right) $ and $B$ is a constant matrix, not necessarily
real.
\end{theorem}

\begin{proof}
The proof can be found in \cite{Brocket}.
\end{proof}

If$\ $we set $t_{0}=0$ then the factorization made in Theorem \ref{theorem1} is reduced
to $\Phi \left( t,0\right) =P^{-1}\left( t\right) e^{Bt}$, this follows from
the fact that $P^{-1}\left( t\right) =\Phi \left( t,0\right) e^{-Bt}$ so if $%
t=0$ then, $P\left( 0\right) =I_{2}$.

Theorem \ref{theorem1} implies the following: for $t>0$, set $t=k\tau +t_{1}$, $t_{1}\in [ 0,\tau ) $ where $k$ is a non-negative integer and $t_{0}=0$,
by the property $\Phi \left( t_{2},t_{0}\right) =\Phi \left(
t_{2},t_{1}\right) \Phi \left( t_{1},t_{0}\right) $ we can write any
solution of (\ref{zzz}) as 
\begin{eqnarray*}
z\left( t\right) &=&\Phi \left( t,0\right) z\left( 0\right) \\
&=&\Phi \left( k\tau +t_{1},0\right) z\left( 0\right) \\
&=&\Phi \left( k\tau +t_{1},\left( k-1\right) \tau \right) \Phi \left(
\left( k-1\right) \tau ,\left( k-2\right) \tau \right) \ldots \Phi \left(
\tau ,0\right) z\left( 0\right)
\end{eqnarray*}%
defining a matrix $M=\Phi \left( \tau ,0\right) $, the solution $z\left(
t\right) $ is 
\begin{equation}
z\left( t\right) =\Phi \left( t_{1},0\right) M^{k}z\left( 0\right)
\label{solu}
\end{equation}%
then, the stability of any solution $z\left( t\right) $ depends on the
matrix $M$, i.e. since $\Phi \left( t_{1},0\right) $ for $t_{1}\in \left[
0,\tau \right) $ and $z\left( 0\right) $ are bounded and using the
well-known fact that: if $\sigma \left( M\right) =\left\{ \rho _{1},\rho
_{2},\ldots ,\rho _{n}\right\} $, then $\sigma \left( M^{k}\right) =\left\{
\rho _{1}^{k},\rho _{2}^{k},\ldots ,\rho _{n}^{k}\right\} $, the only factor of the solution that could grow without bound, as time $t$ increases, is the
matrix $M^{k}$, the matrix $M$ is known as the monodromy matrix associated
to (\ref{zzz}). In fact, from equation (\ref{solu}) we can notice that one
can obtain any solution of (\ref{zzz}) by only knowing the state transition
matrix at the interval $t\in \left[ 0,\tau \right] $. And then, we can state
the following

\begin{lemma} \label{lemma2}
Let $\mathbb{\rho }_{i}$ be the eigenvalues of the monodromy matrix $M$ then
the solutions of (\ref{zzz}) are

\begin{itemize}
\item Asymptotically stable if and only if all $\left\vert \mathbb{\rho }%
_{i}\right\vert <1$

\item Stable if and only if all $\left\vert \mathbb{\rho }_{i}\right\vert
\leq 1$, and if any $\mathbb{\rho }_{i}$ has modulo one, it must be a simple
root of the minimal polynomial of $M$.

\item Unstable if and only if there is a $\mathbb{\rho }_{i}$ such that $%
\left\vert \mathbb{\rho }_{i}\right\vert >1$ or if all $\left\vert \mathbb{%
\rho }_{i}\right\vert \leq 1$ but one $\mathbb{\rho }_{j}:\left\vert \mathbb{%
\rho }_{j}\right\vert =1$ and $\mathbb{\rho }_{j}$ is a multiple root of the
minimal polynomial of $M$.
\end{itemize}
\end{lemma}

\begin{remark} \label{remark3}
Notice that the system (\ref{zzz}) can be written as a Hamiltonian system%
\[
\left[ 
\begin{array}{c}
\dot{z}_{1} \\ 
\dot{z}_{2}%
\end{array}%
\right] =\left[ 
\begin{array}{cc}
0 & 1 \\ 
-1 & 0%
\end{array}%
\right] \left[ 
\begin{array}{cc}
\left( \alpha +\beta p\left( t\right) \right) & 0 \\ 
0 & 1%
\end{array}%
\right] \left[ 
\begin{array}{c}
z_{1} \\ 
z_{2}%
\end{array}%
\right] 
\]%
thus, its solutions cannot be asymptotically stable, they can only be
bounded or unstable. Moreover, the state transition matrix is a symplectic
matrix therefore, its eigenvalues are symmetric with respect to the unitary
circle \cite{meyer}. The condition $\mathbb{\rho }_{i}=\pm 1$ implies that
there exists at least one periodic solution ($\mathbb{\rho }_{i}=1$) or
anti-periodic solution ($\mathbb{\rho }_{i}=-1$).
\end{remark}

We will expand the last part of the latter remark, but first notice that if $%
M\in 
\mathbb{R}
^{2\times 2}$ then its characteristic equation is%
\[
\det \left( \mathbb{\rho }I_{2}-M\right) =\mathbb{\rho }^{2}-Trace\left(
M\right) \mathbb{\rho }+1 
\]%
\qquad\ the independent term is equal to $1$ because of the Liouville
theorem \cite{Lancaster}, and the eigenvalues are%
\[
\mathbb{\rho }_{1,2}=\frac{Trace\left( M\right) \pm \sqrt{\left( Trace\left(
M\right) \right) ^{2}-4}}{2} 
\]%
if we define the discriminant of the equation (\ref{zzz}) as%
\begin{equation}
\Delta \left( \alpha ,\beta \right) =Trace\left( M\right)  \label{dis}
\end{equation}%
then the eigenvalues and therefore the stability of the solutions depend on
the discriminant $\Delta \left( \alpha ,\beta \right) $.

The discriminant of the Hill equation has been extensively studied, see for
example \cite{LyapunovApp,Yak_Str,Hochstadt} among others. One
of the most remarkable theorems on this subject was done by Haupt \cite%
{MagnWin}, he proved that, for fixed $\beta $, the functions $\Delta \left(
\alpha ,\beta \right) =2$ and $\Delta \left( \alpha ,\beta \right) =-2$ have
an infinite number of zeros, and there exist intervals where $\left\vert
\Delta \left( \alpha ,\beta \right) \right\vert $ is less than $2$, some
where $\left\vert \Delta \left( \alpha ,\beta \right) \right\vert $ is
larger than $2$ and some values of $\alpha ,\beta $ where $\left\vert \Delta
\left( \alpha ,\beta \right) \right\vert =2$.

\begin{theorem} \label{theorem4}
Let $\Delta \left( \alpha ,\beta \right) $ be the discriminant of the
periodic differential equation $\ddot{x}+\left( \alpha +\beta p\left(
t\right) \right) x=0$, where $p\left( t\right) $ is a real valued function
and $p\left( t+\tau \right) =p\left( t\right) $, the coefficients $\alpha $
and $\beta $ are real numbers and $\beta $ is fixed. There exist two
infinite sequences 
\[
\lambda _{0}<\lambda _{1}\leq \lambda _{2}<\lambda _{3}\leq \ldots 
\]%
such that $\Delta \left( \lambda _{i},\beta \right) =2$. And 
\[
\lambda _{1}^{\prime }\leq \lambda _{2}^{\prime }<\lambda _{3}^{\prime }\leq
\lambda _{4}^{\prime }<\ldots 
\]%
such that $\Delta \left( \lambda _{i}^{\prime },\beta \right) =-2$. These sequences interlace in such a way that 
\[
\lambda _{0}<\lambda _{1}^{\prime }\leq \lambda _{2}^{\prime }<\lambda
_{1}\leq \lambda _{2}<\lambda _{3}^{\prime }\leq \lambda _{4}^{\prime
}<\lambda _{3}\leq \ldots 
\]%
Whenever $\alpha $ lies in one of the intervals 
\[
\left( \lambda _{0},\lambda _{1}^{\prime }\right) ,\left( \lambda
_{2}^{\prime },\lambda _{1}\right) ,\left( \lambda _{2},\lambda _{3}^{\prime
}\right) ,\left( \lambda _{4}^{\prime },\lambda _{3}\right) ,\ldots ,\text{
then }\left\vert \Delta \left( \alpha ,\beta \right) \right\vert <2 
\]%
if $\alpha $ lies in 
\[
\left( -\infty ,\lambda _{0}\right) ,\left( \lambda _{1}^{\prime },\lambda
_{2}^{\prime }\right) ,\left( \lambda _{1},\lambda _{2}\right) ,\left(
\lambda _{3}^{\prime },\lambda _{4}^{\prime }\right) ,\ldots ,\text{ then }%
\left\vert \Delta \left( \alpha ,\beta \right) \right\vert >2 
\]%
if $\lambda _{k}=\lambda _{k+1}$ then $\left. \Delta \left( \alpha ,\beta
\right) \right\vert _{\alpha =\lambda _{k}}=2$ and $\left. \frac{\partial
\Delta \left( \alpha ,\beta \right) }{\partial \alpha }\right\vert _{\alpha
=\lambda _{k}}=0$; and if $\lambda _{k}^{\prime }=\lambda _{k+1}^{\prime }$
then $\left. \Delta \left( \alpha ,\beta \right) \right\vert _{\alpha
=\lambda _{k}^{\prime }}=-2$ and $\left. \frac{\partial \Delta \left( \alpha
,\beta \right) }{\partial \alpha }\right\vert _{\alpha =\lambda _{k}^{\prime
}}=0$.
\end{theorem}

\begin{proof}
The proof can be seen in \cite{Hochstadt} or \cite{MagnWin}.
\end{proof}

Theorem \ref{theorem4} implies that, for a fixed $\beta $, the discriminant $\Delta
\left( \alpha ,\beta \right) $ of equation (\ref{zzz}) has an infinite
number of stable and unstable intervals, and those intervals will be bounded
by zeros of $\Delta \left( \alpha ,\beta \right) =2$ and $\Delta \left(
\alpha ,\beta \right) =-2$. If we eliminate the condition on $\beta $, the
boundaries characterized by $\Delta \left( \alpha ,\beta \right) =2$ and $%
\Delta \left( \alpha ,\beta \right) =-2$ will define the so called
transition curves in the $\alpha -\beta $ plane, Figure 4 in section 6 shows
the transition curves for the equation 
\[
\ddot{x}+(\alpha +\beta \left( \cos \left( t\right) +\cos \left( 2t\right)
\right) )x=0 
\]

\subsection{Walsh Series}

In this part the main properties of the Walsh functions are described and
some of them are proved.

The Walsh functions $w_{n}\left( t\right) $ form an ordered set of
rectangular waveforms taking only two amplitude values $\pm 1$ and they form an orthogonal set of $\mathcal{L} _{2}\left[ 0,1\right] $, the Lebesgue
space of square integrable functions on $\left[ 0,1\right] $, see \cite%
{Beauchamp}, i.e. 
\[
\int_{0}^{1}hw_{n}\left( t\right) w_{m}\left( t\right) dt=\left\{ 
\begin{array}{c}
h\text{ if }n=m \\ 
0\text{ if }n\not=m%
\end{array}%
\right. 
\]

Walsh functions are defined over a limited time interval $\left[ 0,1\right] $%
, but it may be transformed in any other interval $\left[ a,b\right] $. Two
arguments are required for completing the definition, a time period $t$ and
an ordering number $n$, this number is related to the number of zero crosses
of each Walsh function\footnote{%
When $\sin \left( t\right) $ and $\cos \left( t\right) $ are used \ in a
Fourier series, the index is related to frequency; for Walsh functions, the
index is called sequency and the analysis based on Walsh functions is called
sequency theory.} \cite{Beauchamp}.

\begin{remark} \label{remark5}
The number of Walsh functions considered is always a power of $2$, i.e. $%
\left\{ w_{0},w_{1},\ldots ,w_{2^{k}-1}\right\} $ for some positive integer $%
k.$
\end{remark}

The Walsh functions can be obtained in several different ways such as:
Rademacher functions \cite{Beauchamp} 
\[
w_{n}\left( t\right) =sign\left[ \left( \sin 2\pi t\right)
^{b_{0}}\prod\nolimits_{k=1}^{m}\left( \cos 2^{k}\pi t\right) ^{b_{k}}\right]
\]%
where $b_{0}...b_{m}$ are the binary bits of the number $n$ expressed in
binary i.e. $n=(b_{m}b_{m-1}...b_{0})$. They can also be obtained by Boolean
synthesis and from Hadamard matrices, being the latter the most known and
used. The first eight Walsh functions are shown in Fig. 1%

   \begin{figure}[h]
      \centering
      \includegraphics[trim=40mm 5mm 40mm 3mm,clip,height=2.5in,width=0.6\textwidth]{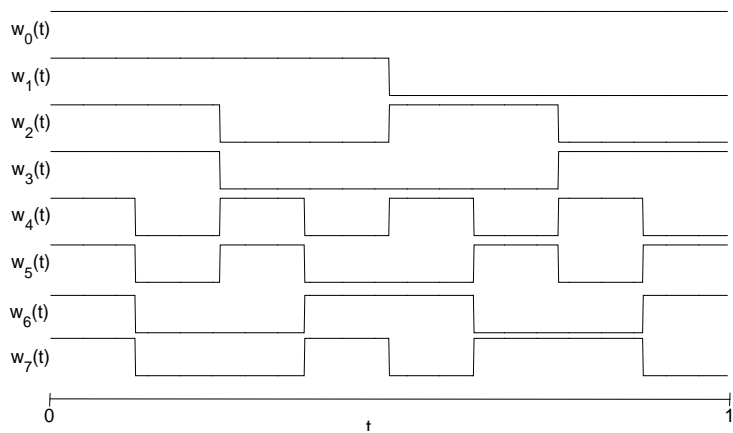}
      \caption{First eight Walsh functions, $\ w_{0}=1$ $\forall t\in \left[ 0,1\right] $. Upper values always are $+1$ and lower values always are $-1$.}
      \label{figurelabel}
   \end{figure}

The representation of Walsh functions by Hadamard matrices is known as Walsh
matrix, we will denote it as $W_{H}$, each row corresponds to a Walsh
function. The Walsh matrix is obtained by rearranging the rows of a Hadamard
matrix\footnote{%
Recall that Hadamard matrices exist only for dimension $2^{k}\times 2^{k}$, $%
k$ being a positive integer. It is consistent with remark \ref{remark5}. The entries of
such matrices are plus and minus ones, and its rows and columns are mutually
orthogonal.}. In this work we are rearranging the rows in a dyadic manner 
\cite{Harmuth}, which is a different ordering of the same set. The Walsh
matrix for the first eight Walsh functions in Fig. 1 is 
\[
W_{H}=\left[ 
\begin{array}{cccccccc}
1 & 1 & 1 & 1 & 1 & 1 & 1 & 1 \\ 
1 & 1 & 1 & 1 & -1 & -1 & -1 & -1 \\ 
1 & 1 & -1 & -1 & 1 & 1 & -1 & -1 \\ 
1 & 1 & -1 & -1 & -1 & -1 & 1 & 1 \\ 
1 & -1 & 1 & -1 & 1 & -1 & 1 & -1 \\ 
1 & -1 & 1 & -1 & -1 & -1 & -1 & 1 \\ 
1 & -1 & -1 & 1 & 1 & -1 & -1 & 1 \\ 
1 & -1 & -1 & 1 & -1 & 1 & 1 & -1%
\end{array}%
\right] 
\]%
the sign of each row entry correspond to the sign of the Walsh function that
is being represented at some time, in other words, the first entry of the
first row is $w_{1}\left( 0\right) $ and the last entry is $w_{1}\left(
1\right) $, $\left( W_{H}\right) _{1,1}=w_{1}\left( 0\right) $ and $\left(
W_{H}\right) _{1,2^{k}}=w_{1}\left( 1\right) $. Notice that $W_{H}\in 
\mathbb{R}
^{2^{k}\times 2^{k}}$ is a symmetric and almost orthogonal matrix, $%
W_{H}^{T}=W_{H}$ and $W_{H}^{-1}W_{H}=2^{k}I_{2^{k}}$, here $^{T}$ is the
transpose operator.

It is known that every function $f\left( t\right) $, which is integrable (in the Lebesgue sense), could be represented by a Walsh series \cite{Harmuth}
defined over the interval $t\in [ 0,1] $ as%
\begin{equation}
f\left( t\right) =\sum_{n=0}^{2^{k}-1}a_{n}w_{n}\left( t\right) =\bar{a}^{T}%
\bar{w}_{2^{k}}\left( t\right)  \label{se}
\end{equation}%
where $\bar{a}^{T}=\left[ 
\begin{array}{ccc}
a_{0}, & ..., & a_{2^{k}-1}%
\end{array}%
\right] \in 
\mathbb{R}
^{2^{k}}$ are the weights associated to each Walsh function and $\bar{w}%
_{2^{k}}\left( t\right) =\left[ 
\begin{array}{ccc}
w_{0}\left( t\right) , & ..., & w_{2^{k}-1}\left( t\right)%
\end{array}%
\right] ^{T}\in 
\mathbb{R}
^{2^{k}}$, $k$ is a positive integer number and it is related to the
accuracy of the representation (\ref{se}). From now on we will use $2^k$ to
denote the order of the approximation. The coefficients $a_{n}$ are given by 
\begin{equation}
a_{n}=\int_{0}^{1}f\left( t\right) w_{n}\left( t\right) dt  \label{co}
\end{equation}%
equation (\ref{se}) can also be written in terms of the Walsh matrix as $%
f\left( t\right) =\bar{a}^{T}W_{H}$.

A very attractive property of Walsh functions is that the integrals of the
function $w_{n}$ which belongs to the finite set $\left\{ w_{0},w_{1},\ldots
,w_{2^{k}-1}\right\} $ can be represented in terms of the set functions \cite%
{Fine}, in order to illustrate this property, lets take the first integral
of $w_{0}\left( t\right) $ and then represent it as in (\ref{se}) for $k=2$
and $t\in [ 0,1] $, so 
\[
\int_{0}^{t}w_{0}\left( \tau_1 \right) d\tau_1 =t 
\]%
if one substitutes $f\left( t\right) =t$ into (\ref{co}) one obtains%
\[
\begin{array}{cc}
a_{0}=\int_{0}^{1}tw_{0}\left( t\right) dt=\frac{1}{2} & a_{1}=%
\int_{0}^{1}tw_{1}\left( t\right) dt=-\frac{1}{4} \\ 
a_{2}=\int_{0}^{1}tw_{2}\left( t\right) dt=-\frac{1}{8} & 
a_{3}=\int_{0}^{1}tw_{3}\left( t\right) dt=0%
\end{array}%
\]%
then the truncated representation (\ref{se}) of the ramp function $f\left(
t\right) =\int_{0}^{t}w_{0}\left( t\right) =t$ is 
\[
\int_{0}^{t}w_{0}\left( t_{1}\right) dt_{1}=\left[ 
\begin{array}{cccc}
\frac{1}{2} & -\frac{1}{4} & -\frac{1}{8} & 0%
\end{array}%
\right] \bar{w}_{4}\left( t\right) 
\]%
similarly, one can take the first integral of $w_{1}\left( t\right) $, $%
w_{2}\left( t\right) $ and $w_{3}\left( t\right) $ to obtain%
\begin{eqnarray*}
\int_{0}^{t}w_{1}\left( t_{1}\right) dt_{1} &=&\left[ 
\begin{array}{cccc}
\frac{1}{4} & 0 & 0 & -\frac{1}{8}%
\end{array}%
\right] \bar{w}_{4}\left( t\right) \\
\int_{0}^{t}w_{2}\left( t_{1}\right) dt_{1} &=&\left[ 
\begin{array}{cccc}
\frac{1}{8} & 0 & 0 & 0%
\end{array}%
\right] \bar{w}_{4}\left( t\right) \\
\int_{0}^{t}w_{3}\left( t_{1}\right) dt_{1} &=&\left[ 
\begin{array}{cccc}
0 & \frac{1}{8} & 0 & 0%
\end{array}%
\right] \bar{w}_{4}\left( t\right)
\end{eqnarray*}%
rewriting the last four equations%
\begin{eqnarray}
\int_{0}^{t}\bar{w}_{4}\left( t_{1}\right) dt_{1} &=&\left[ 
\begin{array}{cccc}
\frac{1}{2} & -\frac{1}{4} & -\frac{1}{8} & 0 \\ 
\frac{1}{4} & 0 & 0 & -\frac{1}{8} \\ 
\frac{1}{8} & 0 & 0 & 0 \\ 
0 & \frac{1}{8} & 0 & 0%
\end{array}%
\right] \bar{w}_{4}\left( t\right)  \label{inW} \\
&=&P^{\left( 4\right) }\bar{w}_{4} (t) \nonumber
\end{eqnarray}%
with similar arguments one can generalize (\ref{inW}) as 
\[
\int_{0}^{t}\bar{w}_{m}\left( t\right) =P^{\left( m\right) }\bar{w}%
_{m}\left( t\right) 
\]%
where $P^{\left( m\right) }\in 
\mathbb{R}
^{m\times m}$, $m=2^{k}$, and 
\begin{eqnarray}
P^{\left( n\right) } &=&\left[ 
\begin{array}{cc}
P^{\left( n-1\right) } & -\frac{1}{2^{n+1}}I_{2^{n-1}} \\ 
\frac{1}{2^{n+1}}I_{2^{n-1}} & 0_{2^{n-1}}%
\end{array}%
\right]  \label{P} \\
P^{\left( 0\right) } &=&\frac{1}{2}  \nonumber
\end{eqnarray}

Matrix $P$ is the Walsh function integration operator. It is worth to notice
that $P$ is non-singular since its determinant is $\det \left( P\right) =%
\frac{1}{2^{k}}$.

Another remarkably useful property of Walsh functions is that they are
closed under multiplication, the multiplication of Walsh functions is
defined as 
\[
w_{n}\left( t\right) w_{m}\left( t\right) =w_{n\oplus m}\left( t\right) 
\]%
where $\oplus $ represents no-carry modulo-2 addition\footnote{%
For example, the no-carry modulo-2 addition of 7 and 2 is equal to 5, since $%
(111)_{B}\oplus (010)_{B}=(101)_{B}$} \cite{Beauchamp}. For example 
\begin{eqnarray*}
w_{0}\left( t\right) w_{5}\left( t\right) &=&w_{5}\left( t\right) \\
w_{2}\left( t\right) w_{6}\left( t\right) &=&w_{4}\left( t\right)
\end{eqnarray*}

It is of great importance to notice that if $w_{n}\left( t\right) $ and $%
w_{m}\left( t\right) $ belong to the finite set of Walsh functions 
\[
\left\{ w_{0}\left( t\right) ,w_{1}\left( t\right) ,\ldots
,w_{2^{k}-1}\left( t\right) \right\} 
\]
then, $w_{n}\left( t\right) w_{m}\left( t\right) $ may be represented in
terms of the finite set of Walsh functions to which they belong. Some other
properties will be presented and proved as we need them. In the next section
we obtain a numerical approximation of the discriminant of the Hill equation.

\section{Approximation of the discriminant $\Delta \left( \protect\alpha ,\protect\beta \right) $}

It is well known that the problem of solving the second order differential
equation, 
\begin{equation}
\ddot{z}+g\left( t\right) z=0\text{, \ \ \ }z\left( 0\right) =a\text{, \ }%
\dot{z}\left( 0\right) =b  \label{eqd}
\end{equation}%
where $g\left( t\right) $ is bounded and for $t\geq 0$, it is equivalent to
solving the associated integral equation \cite{collins} 
\begin{equation}
z=-\int_{0}^{t}\int_{0}^{t_{1}}g\left( t_{2}\right) x\left( t_{2}\right)
dt_{2}dt_{1}+bt+a  \label{eqi}
\end{equation}

There is a vast amount of numerical or pseudo-analytical methods by which
one can approximate the solution of (\ref{eqd}) or (\ref{eqi}), see for
example \cite{Trefethen,He}. In this part we will approximate the
solution by means of Walsh functions i.e. we will propose a solution of the
form $z=\bar{a}^{T}\bar{w}_{m}$ as in (\ref{se}).

Consider the problem of solving a Hill equation 
\begin{equation}
\ddot{x}+\left( \alpha +\beta p\left( t\right) \right) x=0\text{ \ \ \ \ }%
x\left( 0\right) =x_{0},\dot{x}\left( 0\right) =\dot{x}_{0}  \label{Hill}
\end{equation}%
where $p\left( \tau +t\right) =p\left( t\right) $ is a real bounded
function, $\alpha $, $\beta \in 
\mathbb{R}
$ and for $t\in \left[ 0,\tau \right] $. This problem is equivalent to
solving the integral equation 
\begin{equation}
x=-\int_{0}^{t }\int_{0}^{t_{1}}\left( \alpha +\beta p\left( t_{2}\right)
\right) z\left( t_{2}\right) dt_{2}dt_{1}+\dot{x}_{0}t +x_{0},\qquad 0\leqslant t_1\leqslant t\leqslant\tau
\label{inteq}
\end{equation}

We propose a solution of (\ref{inteq}) in terms of Walsh functions 
\begin{equation}
x=\sum_{n=0}^{2^{k}-1}\nu _{n}w_{n}\left( t\right) =\bar{\nu}^{T}\bar{w}%
_{2^{k}}  \label{app}
\end{equation}%
here the constant vector $\bar{\nu}\in 
\mathbb{R}
^{2^{k}}$ is unknown and $\bar{w}_{2^{k}}\left( t\right) $ is the vector of
Walsh functions $\bar{w}_{2^{k}}\left( t\right) =\left[ 
\begin{array}{ccc}
w_{0}\left( t\right) & ... & w_{2^{k}-1}\left( t\right)%
\end{array}%
\right] ^{T}$. Where $2^k$ denotes the
order of the approximation and $k$ is an integer positive number. Integrating (\ref{Hill}) ones we get%
\[
\dot{x}=-\alpha \int_{0}^{t }x\left( t_{1}\right) dt_{1}-\beta
\int_{0}^{t }p\left( t_{1}\right) x\left( t_{1}\right) dt_{1}+\dot{x}_{0} 
\]%
expressing the function $p\left( t\right) $ in terms of Walsh functions, $%
p\left( t\right) =\sum_{n=0}^{2^{k}-1}r_{n}w_{n}\left( t\right) =\bar{r}^{T}%
\bar{w}_{2^{k}}$, where the real constant vector $\bar{r}\in 
\mathbb{R}
^{2^{k}}$ is known. Substituting (\ref{app}) in the latter equation 
\begin{equation}
\dot{x}=-\alpha \int_{0}^{t }\bar{\nu}^{T}\bar{w}_{2^{k}}\left(
t_{1}\right) dt_{1}-\beta \int_{0}^{t }\bar{r}^{T}M_{2^{k}}\bar{\nu}%
dt_{1}+\dot{x}_{0}, \qquad    t \in[0,\tau]  \label{firstint}
\end{equation}%
where the matrix $M_{2^{k}}\triangleq \bar{w}_{2^{k}}\bar{w}_{2^{k}}^{T}$
has the form%
\begin{eqnarray*}
M_{2^{k}} &=&\left[ 
\begin{array}{cccc}
w_{0}w_{0} & w_{1}w_{0} & \cdots & w_{2^{k}-1}w_{0} \\ 
w_{0}w_{1} & w_{1}w_{1} & \cdots & w_{2^{k}-1}w_{0} \\ 
\vdots & \vdots & \vdots & \vdots \\ 
w_{0}w_{2^{k}-1} & w_{1}w_{2^{k}-1} & \cdots & w_{2^{k}-1}w_{2^{k}-1}%
\end{array}%
\right] \\
&=&\left[ 
\begin{array}{cccc}
w_{0\oplus 0} & w_{1\oplus 0} & \cdots & w_{2^{k}-1\oplus 0} \\ 
w_{0\oplus 1} & w_{1\oplus 1} & \cdots & w_{2^{k}-1\oplus 1} \\ 
\vdots & \vdots & \vdots & \vdots \\ 
w_{0\oplus 2^{k}-1} & w_{1\oplus 2^{k}-1} & \cdots & w_{2^{k}-1\oplus
2^{k}-1}%
\end{array}%
\right]
\end{eqnarray*}

Notice that the matrix $M_{2^{k}}$ is a symmetric matrix, moreover, the
first column of matrix $M_{2^{k}}$ is the vector of Walsh functions $\bar{w}%
_{2^{k}}\left( t\right) $ and the other $2^{k}-1$ columns are permutations
of the entries of the first column, that is, the matrix $M_{2^{k}}$ can be
written as 
\begin{equation}
M_{2^{k}}\left( t\right) =\left[ \bar{w}_{2^{k}}\left( t\right) ,\Lambda
_{1}^{\left( 2^{k}\right) }\bar{w}_{2^{k}}\left( t\right) ,...,\Lambda
_{2^{k}-1}^{\left( 2^{k}\right) }\bar{w}_{2^{k}}\left( t\right) \right] 
\label{MMM}
\end{equation}%
where each $\Lambda _{i}^{\left( 2^{k}\right) }$ is a symmetric permutation
matrix, see Lemma in \ref{lemmaA1}.

It is desirable to rewrite the second term on the right hand side of (\ref%
{firstint}) so the vector $\bar{\nu}^{T}$ pre-multiply and $\bar{w}\left(
t\right) $ post-multiply some matrix $Q$, i.e. $\bar{r}^{T}M_{2^{k}}\bar{\nu}%
=\bar{\nu}^{T}Q\bar{w}_{2^{k}}$. For this purpose we introduce the next
Lemma, taken from \cite{Karanam}

\begin{lemma} \label{lemma6}
If $M_{2^{k}}\left( t\right) =\bar{w}_{2^{k}}\bar{w}_{2^{k}}^{T}$ and $%
\gamma \in 
\mathbb{R}
^{2^{k}}$ then $M_{2^{k}}\gamma =\Lambda _{\gamma }\bar{w}_{2^{k}}$ where $%
\Lambda _{\gamma }=\left[ \gamma ,\Lambda _{1}^{\left( 2^{k}\right) }\gamma
,...,\Lambda _{2^{k}-1}^{\left( 2^{k}\right) }\gamma \right] $.
\end{lemma}

\begin{proof}
Multiplying the representation (\ref{MMM}) of $M_{2^{k}}\left( t\right) $ by
a vector $\gamma \in 
\mathbb{R}
^{2^{k}}$ one obtains%
\begin{eqnarray*}
M_{2^{k}}\left( t\right) \gamma &=&\left[ 
\begin{array}{c}
\bar{w}_{2^{k}}^{T}\left( t\right) \gamma \\ 
\bar{w}_{2^{k}}^{T}\left( t\right) \Lambda _{1}^{\left( 2^{k}\right) }\gamma
\\ 
\vdots \\ 
\bar{w}_{2^{k}}^{T}\left( t\right) \Lambda _{2^{k}-1}^{\left( 2^{k}\right)
}\gamma%
\end{array}%
\right] =\left[ 
\begin{array}{c}
\gamma ^{T}\bar{w}_{2^{k}}\left( t\right) \\ 
\gamma ^{T}\Lambda _{1}^{\left( 2^{k}\right) }\bar{w}_{2^{k}}\left( t\right)
\\ 
\vdots \\ 
\gamma ^{T}\Lambda _{2^{k}-1}^{\left( 2^{k}\right) }\bar{w}_{2^{k}}\left(
t\right)%
\end{array}%
\right] \\
&=&\left[ \gamma ,\Lambda _{1}^{\left( 2^{k}\right) }\gamma ,...,\Lambda
_{2^{k}-1}^{\left( 2^{k}\right) }\gamma \right] \bar{w}_{2^{k}}\left(
t\right) \\
&=&\Lambda _{\gamma }\bar{w}_{2^{k}}\left( t\right)
\end{eqnarray*}%
since $M_{2^{k}}\left( t\right) $ and $\Lambda _{n}^{(2^{k})}$ are symmetric
matrices.
\end{proof}

By latter Lemma and $\int_{0}^{t }\bar{w}_{2^{k}}\left( t_{1}\right)
dt_{1}=\tau P\bar{w}_{2^{k}}\left( t\right) $, $t\in[0,\tau]$, $P\triangleq P^{\left(
2^{k}\right) }$, one can rewrite (\ref{firstint}) as 
\begin{eqnarray*}
\dot{x} &=&-\alpha \int_{0}^{t }\bar{\nu}^{T}\bar{w}_{2^{k}}\left(
t_{1}\right) dt_{1}-\beta \int_{0}^{t }\bar{\nu}^{T}\Lambda _{\bar{r}}%
\bar{w}_{2^{k}}\left( t_{1}\right) dt_{1}+\dot{x}_{0} \\
&=&-\tau \left( \alpha \bar{\nu}^{T}P+\beta \bar{\nu}^{T}\Lambda _{\bar{r}%
}P\right) \bar{w}_{2^{k}}\left( t\right) +\dot{x}_{0}
\end{eqnarray*}%
integrating once again, one obtains%
\begin{equation}
x=-\tau ^{2}\left( \alpha \bar{\nu}^{T}+\beta \bar{\nu}^{T}\Lambda _{\bar{r}%
}\right) P^{2}\bar{w}_{2^{k}}\left( t\right) +\tau \dot{x}_{0}e_{1}^{T}P\bar{%
w}_{2^{k}}\left( t\right) +x_{0}  \label{sol}
\end{equation}%
where $e_{1}^{T}=\left[ 
\begin{array}{cccc}
1 & 0 & \cdots & 0%
\end{array}%
\right] $.

From equations (\ref{app}), $x=\bar{\nu}^{T}\bar{w}_{2^{k}}$, and (\ref{sol}%
) one can say 
\[
\bar{\nu}^{T}\bar{w}_{2^{k}}\left( t\right) =\left( -\tau ^{2}\bar{\nu}%
^{T}\left( \alpha I_{2^{k}}+\beta \Lambda _{\bar{r}}\right) P^{2}+\tau \dot{x%
}_{0}e_{1}^{T}P+x_{0}e_{1}^{T}\right) \bar{w}_{2^{k}}\left( t\right) 
\]%
so%
\[
\bar{\nu}^{T}=-\tau ^{2}\bar{\nu}^{T}\left( \alpha I_{2^{k}}+\beta \Lambda _{%
\bar{r}}\right) P^{2}+\tau \dot{x}_{0}e_{1}^{T}P+x_{0}e_{1}^{T} 
\]%
solving for $\bar{\nu}^{T}$%
\[
\bar{\nu}^{T}=\left( \tau \dot{x}_{0}e_{1}^{T}P+x_{0}e_{1}^{T}\right) \Gamma 
\]%
where 
\begin{equation}
\Gamma =\left( I_{2^{k}}+\tau ^{2}\left( \alpha I_{2^{k}}+\beta \Lambda _{%
\bar{r}}\right) P^{2}\right) ^{-1}  \label{Gamma}
\end{equation}%
notice that $\Gamma \in 
\mathbb{R}
^{2^{k}\times 2^{k}}$ is a large dimensional matrix and it does not have any
recognizable pattern.

Then the approximation of the general solution $x$ of \ the Hill equation (%
\ref{Hill}) and its derivative $\dot{x}$, are 
\begin{eqnarray*}
x &=&\left( \tau \dot{x}_{0}e_{1}^{T}P+x_{0}e_{1}^{T}\right) \Gamma \bar{w}%
_{2^{k}}\left( t\right) \\
\dot{x} &=&-\frac{1}{\tau }\left( \tau \dot{x}_{0}e_{1}^{T}P+x_{0}e_{1}^{T}%
\right) \left( I_{2^{k}}-\Gamma \right) P^{-1}\bar{w}_{2^{k}}\left( t\right)
+\dot{x}_{0}
\end{eqnarray*}%
therefore two linearly independent solutions of (\ref{inteq}) are%
\[
\begin{tabular}{lll}
$x_{1}=e_{1}^{T}\Gamma \bar{w}_{2^{k}}\left( t\right) $, & $\text{for }%
x_{0}=1$, & $\dot{x}_{0}=0$ \\ 
$x_{2}=\tau e_{1}^{T}P\Gamma \bar{w}_{2^{k}}\left( t\right) \text{,}$ & $%
\text{for }x_{0}=0\text{,}$ & $\dot{x}_{0}=1$%
\end{tabular}%
\]
and the approximation of the state transition matrix $\Phi \left( t,0\right) 
$ is 
\begin{equation}
\Phi \left( t,0\right) =\left[ 
\begin{array}{cc}
e_{1}^{T}\Gamma \bar{w}_{2^{k}}\left( t\right) & \tau e_{1}^{T}P\Gamma \bar{w%
}_{2^{k}}\left( t\right) \\ 
-\frac{1}{\tau }e_{1}^{T}\left( I_{2^{k}}-\Gamma \right) P^{-1}\bar{w}%
_{2^{k}}\left( t\right) & e_{1}^{T}\left( P\Gamma P^{-1}\right) \bar{w}%
_{2^{k}}\left( t\right)%
\end{array}%
\right] \text{, \ \ \ \ \ }t\in \left[ 0,\tau \right]  \label{STM}
\end{equation}

Finally the approximation of the discriminant (\ref{dis}), $\Delta \left(
\alpha ,\beta \right) =x_{1}\left( \tau \right) +\dot{x}_{2}\left( \tau
\right) $, is 
\begin{equation}
\Delta \left( \alpha ,\beta \right) =e^{T}\left[ \Gamma +P\Gamma P^{-1}%
\right] \bar{w}_{2^{k}}\left( \tau \right)  \label{Dapp}
\end{equation}%
where $\tau $ is the minimum period of the Hill equation excitation
function. The vector $\bar{w}_{2^{k}}\left( \tau \right) $ is the last
column of the Walsh matrix $W_{H}$. The matrices $\Gamma $ and $P\Gamma
P^{-1}$ will be denoted as the discriminant sampling matrices, see remark \ref{remark9}.

Notice that for determining $\Delta \left( \alpha ,\beta \right) $ we only
require the first row of the matrix $\Gamma +P\Gamma P^{-1}$, the main
trouble with (\ref{Dapp}) is that we need to obtain the matrix $\Gamma $, in
other words, we need to find the inverse matrix of a $2^{k}\times 2^{k}$
matrix where $2^k$ is the approximation order and it is related to its
accuracy. Nevertheless, (\ref{Dapp}) is easy to treat as we will see in the
next section, the non-singularity condition of $\Gamma $ will be treated in
section 5.

\section{Simplification of the $\Delta \left( \protect\alpha ,\protect\beta\right) $ approximation}

We had seen that in order to compute $\Delta \left( \alpha ,\beta \right) $, it is necessary to obtain $\Gamma $ which is the inverse of a large
dimensional matrix. In this section we will obtain a simplified version of $%
\Delta \left( \alpha ,\beta \right) $, i.e. we will see that the matrix $%
\Gamma +P\Gamma P^{-1}$ can be rewritten as an upper triangular matrix and
the dependence on Walsh functions will be eliminated.

Let $W_{H}\in 
\mathbb{R}
^{2^{k}\times 2^{k}}$ be the representation of the Walsh functions in terms
of Hadamard matrix, see section 2, using the fact that the Walsh functions
are orthogonal to each other, one can prove that $W_{H}^{-1}=\frac{1}{2^{k}}%
W_{H}^{T}=\frac{1}{2^{k}}W_{H}$.

We claim that the matrix $\Lambda _{\bar{r}}$ is similar to a diagonal
matrix $\bar{\Lambda}_{\bar{r}}\triangleq W_{H}^{-1}\Lambda _{\bar{r}%
}W_{H}=diag\left( p_{1},p_{2},\ldots ,p_{2^{k}}\right) $, where the
constants $p_{n}$ are defined as $p_{n}=p\left( n\frac{\tau }{2^{k}}\right) $%
, $n=1,2,\ldots 2^{k}$, and $p\left( t\right) $ is the Hill equation
excitation function. In addition, we can say that $p_{n}$ are the
eigenvalues of $\Lambda _{\bar{r}}$ and the columns of $W_{H}$ are the
eigenvectors associated to each $p_{n}$, namely, $\Lambda _{\bar{r}%
}W_{H}=W_{H}diag\left( p_{1},p_{2},\ldots ,p_{2^{k}}\right) $. This claim
follows from Lemma \ref{lemma6}, for a detailed proof see \cite{Gulamhusein}. Notice
that the entries of the diagonal matrix $\bar{\Lambda}_{\bar{r}}$ are the
function $p\left( \cdot \right) $ evaluated at the end of the $nth$
interval. The matrix $\bar{\Lambda}_{\bar{r}}$ may be seen as a $p\left(
t\right) $ sampling matrix.

One can prove that the integration operator $P$ is similar to the upper
triangular matrix $\bar{P}\triangleq W_{H}PW_{H}=\frac{1}{2}%
I_{2^{k}}+Q+Q^{2}+\ldots +Q^{2^{k}-1}$ where $Q$ is the nilpotent matrix 
\[
Q=\left[ 
\begin{array}{cccc}
0 & 1 & \cdots  & 0 \\ 
\vdots  & \vdots  & \ddots  & \vdots  \\ 
0 & 0 & \cdots  & 1 \\ 
0 & 0 & \cdots  & 0%
\end{array}%
\right] \in 
\mathbb{R}
^{2^{k}\times 2^{k}}
\]%
and $\frac{\bar{P}^{-1}}{2^{2k+2}}\triangleq $ $W_{H}P^{-1}W_{H}=\frac{1}{2}%
I_{2^{k}}-Q+Q^{2}+\ldots -Q^{2^{k}-1}$, the proof may be found in \cite%
{ChenTsay}. For sake of completeness we write the matrices $\bar{\Lambda}_{%
\bar{r}}$, $\bar{P}$ and $\bar{P}^{-1}$%
\begin{eqnarray*}
\bar{\Lambda}_{\bar{r}} &=&\left[ 
\begin{array}{ccccc}
p_{1} & 0 & \cdots  & 0 & 0 \\ 
0 & p_{2} & \cdots  & 0 & 0 \\ 
\vdots  & \vdots  & \ddots  & \vdots  & \vdots  \\ 
0 & 0 & \cdots  & p_{2^{k}-1} & 0 \\ 
0 & 0 & \cdots  & 0 & p_{2^{k}}%
\end{array}%
\right] \text{, }\bar{P}=\left[ 
\begin{array}{ccccc}
\frac{1}{2} & 1 & \cdots  & 1 & 1 \\ 
0 & \frac{1}{2} & \cdots  & 1 & 1 \\ 
\vdots  & \vdots  & \ddots  & \vdots  & \vdots  \\ 
0 & 0 & \cdots  & \frac{1}{2} & 1 \\ 
0 & 0 & \cdots  & 0 & \frac{1}{2}%
\end{array}%
\right]  \\
\text{and }\bar{P}^{-1} &=&2^{2k+2}\left[ 
\begin{array}{cccccc}
\frac{1}{2} & -1 & 1 & \cdots  & 1 & -1 \\ 
0 & \frac{1}{2} & -1 & \ddots  & -1 & 1 \\ 
0 & 0 & \frac{1}{2} & \ddots  & 1 & -1 \\ 
\vdots  & \vdots  & \vdots  & \ddots  & \vdots  & \vdots  \\ 
0 & 0 & 0 & \cdots  & \frac{1}{2} & -1 \\ 
0 & 0 & 0 & \cdots  & 0 & \frac{1}{2}%
\end{array}%
\right] 
\end{eqnarray*}

\begin{claim} \label{claim7}
The matrix $\Gamma =\left( I_{2^{k}}+\tau ^{2}\left( \alpha I_{2^{k}}+\beta
\Lambda _{R}\right) P^{2}\right) ^{-1}$ is almost orthogonal similar%
\footnote{%
We say that a square matrix $A$ is almost orthogonal similar to $B$ if $%
\exists $ an almost orthogonal matrix $R$ such that $B=RAR^{-1}$ and $%
R^{T}R=\alpha I$, $\alpha \not=0$.} to the upper triangular matrix%
\begin{equation}
\bar{\Gamma}=\left( I_{2^{k}}+\frac{\tau ^{2}}{2^{2k}}\left( \alpha
I_{2^{k}}+\beta \bar{\Lambda}_{\bar{r}}\right) \bar{P}^{2}\right) ^{-1}
\label{GH}
\end{equation}
\end{claim}

\begin{proof}
The proof is based on elementary algebraic properties of matrices, one just
has to pre-multiply and post-multiply $\Gamma $ by $W_{H}^{-1}$ and $W_{H}$
respectively and to use the definitions of $\bar{\Lambda}_{\bar{r}}$, $\bar{P%
}$ and $\bar{P}^{-1}$%
\begin{eqnarray*}
W_{H}^{-1}\Gamma W_{H} &=&W_{H}^{-1}\left( I_{2^{k}}+\tau ^{2}\left( \alpha
I_{2^{k}}+\beta \Lambda _{\bar{r}}\right) P^{2}\right) ^{-1}W_{H} \\
&=&\left( W_{H}^{-1}\left( I_{2^{k}}+\tau ^{2}\left( \alpha I_{2^{k}}+\beta
\Lambda _{\bar{r}}\right) P^{2}\right) W_{H}\right) ^{-1} \\
&=&\left( I_{2^{k}}+\tau ^{2}\alpha W_{H}^{-1}P^{2}W_{H}+\tau ^{2}\beta
W_{H}^{-1}\Lambda _{\bar{r}}P^{2}W_{H}\right) ^{-1} \\
&=&\left( I_{2^{k}}+\tau ^{2}\alpha W_{H}^{-1}PW_{H}W_{H}^{-1}PW_{H}+\tau
^{2}\beta W_{H}^{-1}\Lambda _{\bar{r}}W_{H}W_{H}^{-1}PW_{H}W_{H}^{-1}PW_{H}%
\right) ^{-1} \\
&=&\left( I_{2^{k}}+\frac{\tau ^{2}\alpha }{2^{2k}}W_{H}PW_{H}W_{H}PW_{H}+%
\frac{\tau ^{2}\beta }{2^{2k}}W_{H}^{-1}\Lambda _{\bar{r}%
}W_{H}W_{H}PW_{H}W_{H}PW_{H}\right) ^{-1} \\
&=&\left( I_{2^{k}}+\frac{\tau ^{2}\alpha }{2^{2k}}\bar{P}^{2}+\frac{\tau
^{2}\beta }{2^{2k}}\bar{\Lambda}_{\bar{r}}\bar{P}^{2}\right) ^{-1} \\
&=&\left( I_{2^{k}}+\frac{\tau ^{2}}{2^{2k}}\left( \alpha I+\beta \bar{%
\Lambda}_{\bar{r}}\right) \bar{P}^{2}\right) ^{-1} \\
&=&\bar{\Gamma}
\end{eqnarray*}%
since $\bar{P}$ is an upper triangular matrix; $\bar{P}^{2}$ and $\left(
I_{2^{k}}+\frac{\tau ^{2}}{2^{2k}}\left( \alpha I+\beta \bar{\Lambda}_{\bar{r%
}}\right) \bar{P}^{2}\right) ^{-1}$ are also triangular matrices.
\end{proof}

Similarly, one can prove that the matrix $P\Gamma P^{-1}=\left(
I_{2^{k}}+\tau ^{2}\alpha P^{2}+\tau ^{2}\beta P\Lambda _{\bar{r}}P\right)
^{-1}$ is almost orthogonal similar to an upper triangular matrix $\bar{%
\Gamma}_{P}=W_{H}^{-1}P\Gamma P^{-1}W_{H}$ namely 
\begin{equation}
\bar{\Gamma}_{P}=\left( I_{2^{k}}+\frac{\tau ^{2}\alpha }{2^{2k}}\bar{P}^{2}+%
\frac{\tau ^{2}\beta }{2^{2k}}\bar{P}\bar{\Lambda}_{\bar{r}}\bar{P}\right)
^{-1}  \label{GHP}
\end{equation}

Now we can rewrite the approximation of the discriminant $\Delta \left(
\alpha ,\beta \right) $ as

\begin{lemma} \label{lemma8}
If $\bar{\Gamma}$ and $\bar{\Gamma}_{P}$ are defined as in (\ref{GH}) and (%
\ref{GHP}) respectively, then the approximation of the discriminant $\Delta
\left( \alpha ,\beta \right) $ can be written as%
\begin{equation}
\Delta \left( \alpha ,\beta \right) =e^{T}\left( \bar{\Gamma}+\bar{\Gamma}%
_{P}\right) e_{2^{k}}  \label{Dapp2}
\end{equation}%
where $e=\left[ 
\begin{array}{cccc}
1 & 1 & \ldots & 1%
\end{array}%
\right] ^{T}$ and $e_{2^{k}}=\left[ 
\begin{array}{cccc}
0 & 0 & \ldots & 1%
\end{array}%
\right] ^{T}$.
\end{lemma}

\begin{proof}
It follows from the claim (\ref{claim7})%
\begin{eqnarray*}
\Delta \left( \alpha ,\beta \right) &=&e^{T}W_{H}W_{H}^{-1}\left( \Gamma
+P\Gamma P^{-1}\right) W_{H}W_{H}^{-1}\bar{w}_{2^{k}}\left( \tau \right) \\
&=&e^{T}\left( W_{H}^{-1}\left( \Gamma +P\Gamma P^{-1}\right) W_{H}\right)
e_{2^{k}} \\
&=&e^{T}\left( W_{H}^{-1}\Gamma W_{H}+W_{H}^{-1}P\Gamma P^{-1}W_{H}\right)
e_{2^{k}} \\
&=&e^{T}\left( \bar{\Gamma}+\bar{\Gamma}_{p}\right) e_{2^{k}}
\end{eqnarray*}%
where $e=e_{1}^{T}W_{H}=\left[ 
\begin{array}{cccc}
1 & 1 & \ldots & 1%
\end{array}%
\right] ^{T}$ and $e_{2^{k}}=W_{H}^{-1}\bar{w}_{2^{k}}\left( \tau \right) =%
\left[ 
\begin{array}{cccc}
0 & 0 & \ldots & 1%
\end{array}%
\right] ^{T}$.
\end{proof}

Notice that the discriminant approximation is the summation of the last
column entries of the matrices $\bar{\Gamma}$ and $\bar{\Gamma}_{P}$ and it
does not depend on the Walsh functions any longer.

At first sight the approximations (\ref{Dapp}) and (\ref{Dapp2}) are almost
the same, both of them have the same problem, they depend on the inverse of
large dimensional matrices. But, the fundamental difference between both
approximations is the structure of the matrices $\Gamma $
and $\bar{\Gamma} $, the former is a matrix full of
numbers and seems to have no pattern at all, on the other hand, the latter
is an upper triangular matrix and has a distinguishable pattern, see (\ref%
{MG}). Furthermore, thanks to the triangular form of $\bar{\Gamma}$ and $%
\bar{\Gamma}_{P}$ one can obtain their last column entries, which is done in
the following section.

\begin{remark} \label{remark9}
As we had seen the discriminant is defined as $\Delta \left( \alpha ,\beta
\right) =x_{1}\left( \tau \right) +\dot{x}_{2}\left( \tau \right) $, if we
set $t_{n}=n\frac{\tau }{2^{k}}$ instead of$\ \tau $, and change the
notation of the discriminant so, the dependence on $t_{n}$ be explicit then, 
$\Delta \left( \alpha ,\beta \right) $ may be written as $\Delta \left(
\alpha ,\beta ,t_{n}\right) =x_{1}\left( t_{n}\right) +\dot{x}_{2}\left(
t_{n}\right) $. From the approximation of the state transition matrix (\ref%
{STM}) it is clear that $\Delta \left( \alpha ,\beta ,t_{n}\right)
=e^{T}\left( \bar{\Gamma}+\bar{\Gamma}_{p}\right) e_{n}$, where $e_{n}$ is
equal to a $2^{k}\times 1$ vector of zeros but equal to one on the $n-th$
entry, this follows from the fact $W_{H}^{-1}\bar{w}_{2^{k}}\left(
t_{n}\right) =e_{n}$. So $\Delta \left( \alpha ,\beta ,t_{n}\right) $ gives
us the value of the addition $x_{1}\left( t\right) +\dot{x}_{2}\left(
t\right) $ at $t=t_{n}$, that is, $e^{T}\left( \bar{\Gamma}+\bar{\Gamma}%
_{p}\right) e_{n}$ gives us the sampling of the addition $x_{1}\left(
t\right) +\dot{x}_{2}\left( t\right) $. Thus, the matrices $\bar{\Gamma}$
and $\bar{\Gamma}_{P}$ may be seen as the discriminant sampling matrices.
\end{remark}

\section{Explicit form of the discriminant $\Delta \left( \protect\alpha ,\protect\beta \right) $}

In this part we give an explicit form of the discriminant $\Delta \left(
\alpha ,\beta \right) $ by removing the dependence on the inverse matrices $%
\bar{\Gamma}$ and $\bar{\Gamma}_{P}$, this is done thanks to their
triangular form. First of all we write the matrices $\bar{\Gamma}$ and $\bar{%
\Gamma}_{P}$ in order to see the pattern they follow%
\begin{small}
\begin{equation}
\bar{\Gamma}=\left[ 
\begin{array}{cccccc}
1+\frac{\tau ^{2}\left( \alpha +\beta p_{1}\right) }{2^{2k+2}} & \frac{\tau
^{2}\left( \alpha +\beta p_{1}\right) }{2^{2k}} & \cdots & \left(
2^{k}-3\right) \frac{\tau ^{2}\left( \alpha +\beta p_{1}\right) }{2^{2k}} & 
\left( 2^{k}-2\right) \frac{\tau ^{2}\left( \alpha +\beta p_{1}\right) }{%
2^{2k}} & \left( 2^{k}-1\right) \frac{\tau ^{2}\left( \alpha +\beta
p_{1}\right) }{2^{2k}} \\ 
0 & 1+\frac{\tau ^{2}\left( \alpha +\beta p_{2}\right) }{2^{2k+2}} & \cdots
& \left( 2^{k}-4\right) \frac{\tau ^{2}\left( \alpha +\beta p_{2}\right) }{%
2^{2k}} & \left( 2^{k}-3\right) \frac{\tau ^{2}\left( \alpha +\beta
p_{2}\right) }{2^{2k}} & \left( 2^{k}-2\right) \frac{\tau ^{2}\left( \alpha
+\beta p_{2}\right) }{2^{2k}} \\ 
\vdots & \vdots & \ddots & \vdots & \vdots & \vdots \\ 
0 & 0 & \cdots & 1+\frac{\tau ^{2}\left( \alpha +\beta p_{2^{k}-2}\right) }{%
2^{2k+2}} & \frac{\tau ^{2}\left( \alpha +\beta p_{2^{k}-2}\right) }{2^{2k}}
& 2\frac{\tau ^{2}\left( \alpha +\beta p_{2^{k}-2}\right) }{2^{2k}} \\ 
0 & 0 & \cdots & 0 & 1+\frac{\tau ^{2}\left( \alpha +\beta
p_{2^{k}-1}\right) }{2^{2k+2}} & \frac{\tau ^{2}\left( \alpha +\beta
p_{2^{k}-1}\right) }{2^{2k}} \\ 
0 & 0 & \cdots & 0 & 0 & 1+\frac{\tau ^{2}\left( \alpha +\beta
p_{2^{k}}\right) }{2^{2k+2}}%
\end{array}%
\right] ^{-1}  \label{MG}
\end{equation}
\end{small}
\begin{scriptsize}
\[
\bar{\Gamma}^{-1}_{P}=\left[ 
\begin{array}{ccccc}
1+\frac{\tau ^{2}\left( \alpha +\beta p_{1}\right) }{2^{2k+2}} & \frac{\tau
^{2}\left( \alpha +\frac{\beta }{2}\left( p_{1}+p_{2}\right) \right) }{%
2^{2k+2}} & \cdots & \frac{\tau ^{2}\left( \left( 2^{k}-2\right) \alpha
+\beta \left( \frac{1}{2}p_{1}+p_{2}+\ldots +\frac{1}{2}p_{2^{k}-1}\right)
\right) }{2^{2k+2}} & \frac{\tau ^{2}\left( \left( 2^{k}-1\right) \alpha
+\beta \left( \frac{1}{2}p_{1}+p_{2}+\ldots +\frac{1}{2}p_{2^{k}}\right)
\right) }{2^{2k+2}} \\ 
0 & 1+\frac{\tau ^{2}\left( \alpha +\beta p_{2}\right) }{2^{2k+2}} & \cdots
& \frac{\tau ^{2}\left( \left( 2^{k}-3\right) \alpha +\beta \left( \frac{1}{2%
}p_{2}+p_{3}+\ldots +\frac{1}{2}p_{2^{k}-1}\right) \right) }{2^{2k+2}} & 
\frac{\tau ^{2}\left( \left( 2^{k}-2\right) \alpha +\beta \left( \frac{1}{2}%
p_{2}+p_{3}+\ldots +\frac{1}{2}p_{2^{k}}\right) \right) }{2^{2k+2}} \\ 
\vdots & \vdots & \ddots & \vdots & \vdots \\ 
0 & 0 & \cdots & \frac{\tau ^{2}\left( \alpha +\frac{\beta }{2}\left(
p_{2^{k}-2}+p_{2^{k}-1}\right) \right) }{2^{2k+2}} & \frac{\tau ^{2}\left(
2\alpha +\beta \left( \frac{1}{2}p_{2^{k}-2}+p_{2^{k}-1}+\frac{1}{2}%
p_{2^{k}}\right) \right) }{2^{2k+2}} \\ 
0 & 0 & \cdots & 1+\frac{\tau ^{2}\left( \alpha +\beta p_{2^{k}-1}\right) }{%
2^{2k+2}} & \frac{\tau ^{2}\left( \alpha +\frac{\beta }{2}\left(
p_{2^{k}-1}+p_{2^{k}}\right) \right) }{2^{2k+2}} \\ 
0 & 0 & \cdots & 0 & 1+\frac{\tau ^{2}\left( \alpha +\beta p_{2^{k}}\right) 
}{2^{2k+2}}%
\end{array}%
\right]  
\]
\end{scriptsize}
Now that we have explicitly written the matrices $\bar{\Gamma}$ and $\bar{%
\Gamma}_{p}$, we can deal with the non-singularity condition.

\begin{remark} \label{remark10}
By simple inspection of (\ref{MG}) we can notice that the only chance for
matrices $\bar{\Gamma}$ and $\bar{\Gamma}_{p}$ to be singular is that the
function $q\left( t\right) \triangleq \alpha +\beta p\left( t\right) $ be
large enough, in modulo, that some entry of the main diagonal be equal to
zero. In other words for $\bar{\Gamma}$ and $\bar{\Gamma}_{p}$ to be
singular the equality 
\begin{equation}
\alpha +\beta p_{n}=-\frac{2^{2k+2}}{\tau ^{2}}  \label{Sing_cond}
\end{equation}%
must be fulfilled for some $t_{n}=\frac{n\tau }{2^{k}}$, $n=1,2,\ldots ,2^{k}
$. Remembering that $\bar{\Gamma}$ and $\bar{\Gamma}_{P}$ are $2^{k}\times
2^{k}$ real matrices and $2^k$ is the approximation order, (\ref{Sing_cond})
implies that for $\alpha $, $\beta $ and $p\left( t_{k}\right) $ small
enough and large enough matrices the non-singularity is guaranteed.
\end{remark}

As we said, in the previous section, the discriminant approximation (\ref%
{Dapp2}) computation just requires the last column of the matrices $\bar{%
\Gamma}$ and $\bar{\Gamma}_{P}$, from now on we will call $b_{n}$ and $c_{n}$
($n=0,1,\ldots 2^{k}-1$) to the $2^{k}-n$ entry of the last column of $\bar{%
\Gamma}$ and $\bar{\Gamma}_{P}$ respectively, i.e. 
\[
\bar{\Gamma}=\left[ 
\begin{array}{ccccc}
\ast & \ast & \cdots & \ast & b_{2^{k}-1} \\ 
0 & \ast & \cdots & \ast & b_{2^{k}-2} \\ 
\vdots & \vdots & \ddots & \vdots & \vdots \\ 
0 & 0 & \cdots & \ast & b_{1} \\ 
0 & 0 & \cdots & 0 & b_{0}%
\end{array}%
\right] \text{ and \ \ \ }\bar{\Gamma}_{P}=\left[ 
\begin{array}{ccccc}
\ast & \ast & \cdots & \ast & c_{2^{k}-1} \\ 
0 & \ast & \cdots & \ast & c_{2^{k}-2} \\ 
\vdots & \vdots & \ddots & \vdots & \vdots \\ 
0 & 0 & \cdots & \ast & c_{1} \\ 
0 & 0 & \cdots & 0 & c_{0}%
\end{array}%
\right] 
\]%
with this definitions the approximation of the discriminant can be written
as 
\begin{equation}
\Delta \left( \alpha ,\beta \right) =\sum_{n=0}^{2^{k}-1}\left(
b_{n}+c_{n}\right)  \label{summ}
\end{equation}%
the Lemma in \ref{lemmaA2} gives us a recursive method for obtaining the
entries of the last column of a non-singular upper triangular matrix.

Now we are ready to state and prove the first of the two main results of
this work, Theorem \ref{theorem11} gives us, in an explicit manner, the approximation of
the discriminant $\Delta \left( \alpha ,\beta \right) $ as a recursive
summation, in other words, we eliminate the dependence on the inverse
matrices $\bar{\Gamma}$ and $\bar{\Gamma}_{P}$.

\subsection{Eliminating the dependence on inverse matrices}

The following Lemma gives us a recursive method to obtain the coefficients $%
b_{n}$ and $c_{n}$, so the discriminant may be written as in (\ref{summ}).

\begin{theorem} \label{theorem11}
If the discriminant sampling matrices $\bar{\Gamma}$ and $\bar{\Gamma}_{P}$
are defined as in (\ref{GH}) and (\ref{GHP}) respectively then, the entries $%
b_{n}$ and $c_{n}$ are:%
\begin{eqnarray}
b_{0} &=&c_{0}=\frac{2^{2k+2}}{2^{2k+2}+\tau ^{2}\left( \alpha +\beta
p_{2^{k}}\right) }  \nonumber \\
b_{n} &=&-\psi _{n}\xi _{n}\sum_{i=0}^{n-1}S_{i}  \label{B} \\
c_{n} &=&-\psi _{n}\sum\limits_{i=0}^{n-1}\left( c_{i}\sum_{j=i+1}^{n}\mu
_{j}\right)   \label{C}
\end{eqnarray}%
for $n=1,2,\ldots ,2^{k}-1$ and with 
\begin{eqnarray*}
S_{h} &=&\sum_{i=0}^{h}b_{i} \\
Z_{h} &=&\sum_{i=0}^{h}c_{i} \\
\psi _{h} &=&\frac{4\tau ^{2}}{2^{2k+2}+\tau ^{2}\left( \alpha +\beta
p_{2^{k}-h}\right) } \\
\xi _{h} &=&\alpha +\beta p_{2^{k}-h} \\
\mu _{h} &=&\alpha +\frac{\beta }{2}\left( p_{2^{k}-h}+p_{2^{k}-h+1}\right) 
\end{eqnarray*}%
And the discriminant approximation is 
\begin{equation}
\Delta \left( \alpha ,\beta \right) =S_{2^{k-1}}+Z_{2^{k-1}}  \label{summ2}
\end{equation}
\end{theorem}

\begin{proof}
Defining $b_{n}$, $n=0,1,...2^{k}-1$, as the $2^{k}-n$ entry of the last
column of $\bar{\Gamma}$ and by direct application of the Lemma  in \ref{lemmaA2} 
on matrix $\bar{\Gamma}$, eq. (\ref{MG}), and doing some simple
algebraic operations one obtains%
\begin{eqnarray*}
b_{0} &=&\frac{2^{2k+2}}{2^{2k+2}+\tau ^{2}\left( \alpha +\beta
p_{2^{k}}\right) }  \nonumber \\
b_{1} &=&-b_{0}\frac{\frac{\tau ^{2}\left( \alpha +\beta p_{2^{k}-1}\right) 
}{2^{2k}}}{1+\frac{\tau ^{2}\left( \alpha +\beta p_{2^{k}-1}\right) }{%
2^{2k+2}}}  \nonumber \\
&=&4\tau ^{2}\frac{\alpha +\beta p_{2^{k}-1}}{2^{2k+2}+\tau ^{2}\left(
\alpha +\beta p_{2^{k}-1}\right) }\left( -b_{0}\right)   \nonumber \\
b_{2} &=&\frac{-b_{1}\frac{\tau ^{2}\left( \alpha +\beta p_{2^{k}-2}\right) 
}{2^{2k}}-2b_{0}\frac{\tau ^{2}\left( \alpha +\beta p_{2^{k}-2}\right) }{%
2^{2k}}}{1+\frac{\tau ^{2}\left( \alpha +\beta p_{2^{k}-2}\right) }{2^{2k+2}}%
}  \nonumber \\
&=&\frac{\frac{\tau ^{2}\left( \alpha +\beta p_{2^{k}-2}\right) }{2^{2k}}}{%
\frac{2^{2k+2}+\tau ^{2}\left( \alpha +\beta p_{2^{k}-2}\right) }{2^{2k+2}}}%
\left( -b_{1}-2b_{0}\right)   \nonumber \\
&=&4\tau ^{2}\frac{\alpha +\beta p_{2^{k}-2}}{2^{2k+2}+\tau ^{2}\left(
\alpha +\beta p_{2^{k}-2}\right) }\left( -b_{1}-2b_{0}\right)   \nonumber 
\end{eqnarray*}%

\begin{eqnarray}
b_{3} &=&4\tau ^{2}\frac{\alpha +\beta p_{2^{k}-3}}{2^{2k+2}+\tau ^{2}\left(
\alpha +\beta p_{2^{k}-3}\right) }\left( -b_{2}-2b_{1}-3b_{0}\right)  
\nonumber \\
&&\vdots   \nonumber \\
b_{n} &=&4\tau ^{2}\frac{\alpha +\beta p_{2^{k}-n}}{2^{2k+2}+\tau ^{2}\left(
\alpha +\beta p_{2^{k}-n}\right) }\left( -b_{n-1}-\ldots -\left( n-1\right)
b_{2}-nb_{0}\right)   \label{b}
\end{eqnarray}%
If we define 
\begin{eqnarray}
S_{h} &\triangleq &\sum_{i=0}^{h}b_{i}  \label{Sh} \\
\psi _{h} &\triangleq &\frac{4\tau ^{2}}{2^{2k+2}+\tau ^{2}\left( \alpha
+\beta p_{2^{k}-h}\right) }  \label{Psih} \\
\xi _{h} &\triangleq &\alpha +\beta p_{2^{k}-h}  \label{Xih}
\end{eqnarray}%
then (\ref{b}) becomes%
\begin{eqnarray*}
b_{0} &=&\frac{2^{2k+2}}{2^{2k+2}+\tau ^{2}\left( \alpha +\beta
p_{2^{k}}\right) } \\
b_{n} &=&-\psi _{n}\xi _{n}\sum_{i=0}^{n-1}S_{i}
\end{eqnarray*}%
for $n=1,2,\ldots 2^{k}-1$. By a similar procedure we can obtain the
formulas for the coefficients $c_{n}$, see appendix C%
\begin{eqnarray*}
c_{0} &=&\frac{2^{2k+2}}{2^{2k+2}+\tau ^{2}\left( \alpha +\beta
p_{2^{k}}\right) } \\
c_{n} &=&-\psi _{n}\sum\limits_{i=0}^{n-1}\left( c_{i}\sum_{j=i+1}^{n}\mu
_{j}\right)  \\
Z_{h} &\triangleq &\sum_{i=0}^{h}c_{i}
\end{eqnarray*}%
where 
\begin{equation}
\mu _{h}\triangleq \alpha +\frac{\beta }{2}\left(
p_{2^{k}-h}+p_{2^{k}-h+1}\right)   \label{Muh}
\end{equation}%
for $n=1,2,\ldots 2^{k}-1$, $\psi _{h}$ is defined as in (\ref{Psih}).

The last statement of the theorem follows from the equation (\ref{summ}) and
the definitions of the summations $S_{h}$ and $Z_{h}$
\end{proof}

Theorem \ref{theorem11} gives us a recursive method for obtaining the coefficients $b_{n}$%
, $c_{n}$ and the approximation of $\Delta \left( \alpha ,\beta \right) $.
Notice that if we define the function $q\left( t\right) \triangleq \alpha
+\beta p\left( t\right) $ then, the coefficients will depend on the
summation of $q\left( t_{n}\right) $ at $t_{n}=n\frac{\tau }{2^{k}}$ over a
subset $J^{\prime }$ of $J=\left\{ 1,2,\ldots ,2^{k}\right\} $, i.e. $b_{n}$
and $c_{n}$ depend on $\sum_{n\in J^{\prime }}q\left( t_{n}\right) $ where $%
J^{\prime }\subset J$. Now, if we do $2^{k}\rightarrow \infty $ then $\frac{%
\tau }{2^{k}}\sum_{n\in J^{\prime }}q\left( t_{n}\right) $ behaves like a
definite integral. And the next questions, arise: Could the coefficients $%
b_{n}$ and $c_{n}$ be written in terms of definite integrals? Moreover,
Could the summation of the coefficients $b_{n}$ and $c_{n}$ be written as a
summation of definite integrals? Next corollary gives the affirmative
answer to the latter question.

\begin{corollary} \label{corollary12}
If the order $2^k$ of the approximation (\ref{summ2}) is large enough so, $%
2^{k}\rightarrow \infty $ and $2^{2k+2}>>\tau ^{2}\left( \alpha +\beta
p_{2^{k}}\right) $ then, the summations of the first $n$ coefficients $b_{n}$
and $c_{n}$, $S_{n}=\sum_{i=0}^{n}b_{n}$ and $Z_{n}=\sum_{i=0}^{n}c_{n}$,
are 
\begin{eqnarray}
S_{0} &=&Z_{0}\approx 1  \nonumber \\
S_{n} &\approx &1-\delta \sum_{i=0}^{n-1}\left[ S_{i}\int\limits_{\tau
-\left( n+1\right) \delta }^{\tau -\left( i+1\right) \delta }\left( \alpha
+\beta p\left( t\right) \right) dt\right]  \label{S2} \\
Z_{n} &\approx &1-\delta \left[ \sum_{i=0}^{n-1}\left( n-i\right)
Z_{i}\int_{\tau -\left( i+2\right) \delta }^{\tau -\left( i+1\right) \delta
}\left( \alpha +\beta p\left( t\right) \right) dt\right]  \label{Z2}
\end{eqnarray}
\end{corollary}

\begin{proof}
As we are assuming that $2^{2k+2}>>\tau ^{2}\left( \alpha +\beta
p_{2^{n}}\right) $ then (\ref{Psih}) becomes%
\begin{equation}
\delta ^{2}\triangleq \psi _{h}\approx \frac{\tau ^{2}}{2^{2k}}  \label{de}
\end{equation}%
moreover%
\[
b_{0}=c_{0}\approx 1 
\]

For the first part of the corollary we must notice that the summation $%
S_{n}=\sum_{i=0}^{n}b_{n}$ can be written as%
\begin{equation}
S_{n}\approx 1-\delta ^{2}\sum_{i=0}^{n-1}\left[ S_{i}\sum%
\limits_{j=i+1}^{n}\xi _{j}\right]  \label{SSb}
\end{equation}%
it follows since%
\begin{eqnarray*}
b_{0} &\approx &1 \\
b_{1} &\approx &-\delta ^{2}\xi _{1}S_{0} \\
b_{2} &\approx &-\delta ^{2}\xi _{2}\left( S_{0}+S_{1}\right) \\
&&\vdots \\
b_{n} &\approx &-\delta ^{2}\xi _{n}\left( S_{0}+S_{1}+\ldots
+S_{n-2}+S_{n-1}\right)
\end{eqnarray*}%
adding the coefficients $b_{0}$ to $b_{n}$ and grouping terms we have 
\[
\sum_{i=0}^{n}b_{n}\approx 1-\delta ^{2}\left[ S_{0}\left( \xi _{1}+\ldots
+\xi _{n}\right) +S_{1}\left( \xi _{2}+\ldots +\xi _{n}\right) +\ldots
+S_{n-1}\left( \xi _{n}\right) \right] 
\]%
thus (\ref{SSb}) follows.

From the definition of $\xi _{h}$ and $\delta $, equations (\ref{Xih}) and (%
\ref{de}) respectively, one can notice that for $\ell <n$ 
\[
\lim_{\delta \rightarrow 0}\delta \sum_{i=\ell }^{n}\xi _{i}\approx
\int\limits_{\tau -\left( n+1\right) \delta }^{\tau -\left( \ell +1\right)
\delta }\left( \alpha +\beta p\left( t\right) \right) dt 
\]%
then, we can rewrite (\ref{SSb}) as a summation that depends on integrals 
\[
S_{n}\approx 1-\delta \sum_{i=0}^{n-1}\left[ S_{i}\int\limits_{\tau -\left(
n+1\right) \delta }^{\tau -\left( i+1\right) \delta }\left( \alpha +\beta
p\left( t\right) \right) dt\right] 
\]

For the second part of the corollary, we directly apply the definite integral definition to (\ref{C}) and one obtains 
\begin{eqnarray*}
c_{n} &\approx &-\delta ^{2}\sum\limits_{i=0}^{n-1}\left(
c_{i}\sum_{j=i+1}^{n}\mu _{j}\right) \\
&=&-\delta \sum\limits_{j=0}^{n-1}\left( c_{j}\int\limits_{\tau -\left(
n+1\right) \delta }^{\tau -\left( j+1\right) \delta }\left( \alpha +\beta
p\left( t\right) \right) dt\right)
\end{eqnarray*}%
thus, if $Z_{n}=\sum_{i=0}^{n}c_{n}$%
\begin{eqnarray*}
Z_{0} &\approx &1 \\
Z_{1} &\approx &1-\delta \left( c_{0}\int\limits_{\tau -2\delta }^{\tau
-\delta }\left( \alpha +\beta p\left( t\right) \right) dt\right) \\
Z_{2} &\approx &1-\delta \left( c_{0}\int\limits_{\tau -2\delta }^{\tau
-1\delta }\left( \alpha +\beta p\left( t\right) \right)
dt+c_{0}\int\limits_{\tau -3\delta }^{\tau -1\delta }\left( \alpha +\beta
p\left( t\right) \right) dt+c_{1}\int\limits_{\tau -3\delta }^{\tau -2\delta
}\left( \alpha +\beta p\left( t\right) \right) dt\right) \\
&=&1-\delta \left( 2c_{0}\int\limits_{\tau -2\delta }^{\tau -1\delta }\left(
\alpha +\beta p\left( t\right) \right) dt+\left( c_{0}+c_{1}\right)
\int\limits_{\tau -3\delta }^{\tau -2\delta }\left( \alpha +\beta p\left(
t\right) \right) dt\right) \\
Z_{3} &\approx&1-\delta \left( 3c_{0}\int\limits_{\tau -2\delta }^{\tau
-1\delta }\left( \alpha +\beta p\left( t\right) \right) dt+2\left(
c_{0}+c_{1}\right) \int\limits_{\tau -3\delta }^{\tau -2\delta }\left(
\alpha +\beta p\left( t\right) \right) dt+\left( c_{0}+c_{1}+c_{2}\right)
\int\limits_{\tau -4\delta }^{\tau -3\delta }\left( \alpha +\beta p\left(
t\right) \right) dt\right) \\
&&\vdots \\
Z_{n} &\approx &1-\delta \left[ \sum_{i=0}^{n-1}\left( n-i\right)
Z_{i}\int\limits_{\tau -\left( i+2\right) \delta }^{\tau -\left( i+1\right)
\delta }\left( \alpha +\beta p\left( t\right) \right) dt\right]
\end{eqnarray*}
\end{proof}

Corollary \ref{corollary12} not only gives us a recursive method to obtain the summation of
the coefficients $b_{n}$ and $c_{n}$, but, by doing some considerations, it
transforms the dependence of the "discrete" approximation of Theorem \ref{theorem11} on
some discrete values of the excitation function $p\left( t\right) $ into a
dependence on the definite integral of the continuous function $p\left(
t\right) $. We should notice that the new expressions of the summations of the
coefficients $b_{n}$ and $c_{n}$ ($S_{n}$ and $Z_{n}$ respectively) depend
on $\delta $ and on the summation of a large number of definite integrals, so
it seems that the discriminant approximation $\Delta \left( \alpha ,\beta
\right) \approx $ $S_{2^{k}-1}+Z_{2^{k}-1}$ can be reduced even more.

Now we are ready to prove the second main result of this work. In the next
Theorem we give an alternative proof of the discriminant approximation
obtained by Lyapunov in his outstanding work \cite{Lyapunov}, see Section 1.
This new proof is based on Theorem \ref{theorem11} and corollary \ref{corollary12}, and is completely
independent of the proof made by Lyapunov.

\begin{theorem} \label{theorem13}
If $\Delta \left( \alpha ,\beta \right) $ is the discriminant of a second
order periodic differential equation%
\[
\ddot{x}+q\left( t\right) x=0\text{, \ \ }q\left( \tau +t\right) =q\left(
t\right) 
\]%
then, $\Delta \left( \alpha ,\beta \right) $ can be expressed as an
alternating series 
\[
\Delta \left( \alpha ,\beta \right) =2-A_{1}+A_{2}+\ldots +\left( -1\right)
^{n}A_{n} 
\]%
where the constants $A_{n}$ $n=1,2,\ldots ,2^{k}-1$ are defined as the
multiple integrals%
\begin{eqnarray*}
A_{0} &=&2,\text{ \ \ \ \ \ \ \ \ \ \ }A_{1}=\tau \int_{0}^{\tau }q\left(
t_{1}\right) dt_{1} \\
A_{2} &=&\int_{0}^{\tau }dt_{1}\int_{0}^{t_{1}}\left( \tau
-t_{1}+t_{2}\right) \left( t_{1}-t_{2}\right) q\left( t_{1}\right) q\left(
t_{2}\right) dt_{2} \\
&&\vdots \\
A_{n} &=&\int_{0}^{\tau }dt_{1}\int_{0}^{t_{1}}dt_{2}\ldots
\int_{0}^{t_{n-1}}\left( \tau -t_{1}+t_{n}\right) \left( t_{1}-t_{2}\right)
\ldots \left( t_{n-1}-t_{n}\right) q\left( t_{1}\right) q\left( t_{2}\right)
\ldots q\left( t_{n}\right) dt_{n}
\end{eqnarray*}
\end{theorem}

\begin{proof}
We know that $\Delta \left( \alpha ,\beta \right) =S_{2^{k}-1}+Z_{2^{k}-1}$,
moreover, corollary \ref{corollary12} gives us recursive formulas to obtain $S_{n}$ and $%
Z_{n}$. The proof of this theorem is based on rewriting the formulas for $%
S_{n}$ and $Z_{n}$, this rewriting must replace the recursion by an
expansion in terms of powers of the parameter $\delta =\frac{\tau }{2^{k}}$.
If we define 
\[
\mathbb{I}_{n,m}\triangleq \int_{\tau -n\delta }^{\tau -m\delta }q\left(
t_{1}\right) dt_{1} 
\]%
then, $S_{n}$ and $Z_{n}$, equation (\ref{S2}) and (\ref{Z2}) respectively,
are 
\begin{eqnarray*}
S_{0} &=&Z_{0}=1 \\
S_{n} &=&1-\delta \sum_{i=0}^{n-1}\left[ S_{i}\mathbb{I}_{n+1,i+1}\right] \\
Z_{n} &=&1-\delta \left[ \sum_{i=0}^{n-1}\left( n-i\right) Z_{i}\mathbb{I}%
_{i+2,i+1}\right]
\end{eqnarray*}

Following the formula for $S_{n}$ and $Z_{n}$ and grouping terms of powers
of $\delta $ one obtains, see Lemma in \ref{lemmaA41}%
\begin{eqnarray*}
S_{n} &=&1-\delta \sum_{i=1}^{n}\mathbb{I}_{n+1,i}+\delta ^{2}\sum_{i=2}^{n}%
\mathbb{I}_{n+1,i}\sum_{j=1}^{i-1}\mathbb{I}_{i,j}-\delta ^{3}\sum_{i=3}^{n}%
\mathbb{I}_{n+1,i}\sum_{j=1}^{i-1}\mathbb{I}_{i,j}\sum_{l=1}^{j}\mathbb{I}%
_{j,l}+\ldots \\
&& \quad\quad\quad\quad\quad\quad\quad\quad\quad\quad\quad\quad\quad\quad\quad\quad\quad\quad\quad\quad\quad\ldots+\left( -1\right) ^{\bar{n}}\delta ^{\bar{n}}\sum_{i=\bar{n}%
}^{n}\mathbb{I}_{n+1,i}\sum_{j=1}^{i-1}\mathbb{I}_{i,j}\sum_{l=1}^{j}\mathbb{%
I}_{j,l}\ldots \sum_{y=1}^{x}\mathbb{I}_{x,y} \\
Z_{n} &=&1-\delta \sum_{i=2}^{n+1}\mathbb{I}_{i,1}+\delta ^{2}\sum_{i=2}^{n}%
\mathbb{I}_{i,1}\sum_{j=i+1}^{n+1}\mathbb{I}_{j,i}-\delta
^{3}\sum_{i=2}^{n-1}\mathbb{I}_{i,1}\sum_{j=i+1}^{n}\mathbb{I}%
_{j,i}\sum_{l=j+1}^{n+1}\mathbb{I}_{l,j}+\ldots\\
&& \quad\quad\quad\quad\quad\quad\quad\quad\quad\quad\quad\quad\quad\quad\quad\quad\quad\quad\quad\quad\quad\quad \ldots+\left( -1\right) ^{\bar{n}%
}\delta ^{\bar{n}}\sum_{i=2}^{n+2-\bar{n}}\mathbb{I}_{i,1}\sum_{j=i+1}^{n+3-%
\bar{n}}\mathbb{I}_{j,i}\ldots \sum_{y=x+1}^{n+1}\mathbb{I}_{y,x}
\end{eqnarray*}

As we have said the discriminant $\Delta \left( \alpha ,\beta \right) $ is
approximately equal to the addition of $S_{2^{k}-1}$ and $Z_{2^{k}-1}$, then 
$\Delta \left( \alpha ,\beta \right) $ may be written as an expansion of
powers of $\delta $. If we define the coefficients $A_{n}$, $n=0,1,2,\ldots
,2^{k}-1$, as the terms associated to the $n-th$ power of $\delta $, i.e. 
\begin{eqnarray}
A_{0} &=&2  \nonumber \\
A_{1} &=&\delta \left( \sum_{i=1}^{2^{k}-1}\mathbb{I}_{2^{k},i}+%
\sum_{i=2}^{2^{k}}\mathbb{I}_{i,1}\right)   \nonumber \\
A_{2} &=&\delta ^{2}\left[ \sum_{i=2}^{2^{k}-1}\left( \mathbb{I}%
_{2^{k},i}\sum_{j=1}^{i-1}\mathbb{I}_{i,j}\right)
+\sum_{i=2}^{2^{k}-1}\left( \mathbb{I}_{i,1}\sum_{j=i+1}^{2^{k}-1+1}\mathbb{I%
}_{j,i}\right) \right]   \label{AA} 
\end{eqnarray}%
\begin{eqnarray*}
A_{3} &=&\delta ^{3}\left[ \sum_{i=3}^{n}\left( \mathbb{I}%
_{n+1,i}\sum_{j=1}^{i-1}\mathbb{I}_{i,j}\sum_{l=1}^{j}\mathbb{I}%
_{j,l}\right) +\sum_{i=2}^{n-1}\left( \mathbb{I}_{i,1}\sum_{j=i+1}^{n}%
\mathbb{I}_{j,i}\sum_{l=j+1}^{n+1}\mathbb{I}_{l,j}\right) \right]   \nonumber
\\
&&\vdots   \nonumber \\
A_{n} &=&\delta ^{n}\left[ \sum_{i=n}^{2^{k}-1}\left( \mathbb{I}%
_{2^{k},i}\sum_{j=1}^{i-1}\mathbb{I}_{i,j}\sum_{l=1}^{j}\mathbb{I}%
_{j,l}\ldots \sum_{y=1}^{x}\mathbb{I}_{x,y}\right)
+\sum_{i=2}^{2^{k}+1-n}\left( \mathbb{I}_{i,1}\sum_{j=i+1}^{2^{k}+2-n}%
\mathbb{I}_{j,i}\ldots \sum_{y=x+1}^{2^{k}}\mathbb{I}_{y,x}\right) \right]  
\nonumber
\end{eqnarray*}%
then the discriminant is 
\[
\Delta \left( \alpha ,\beta \right) =A_{0}-A_{1}+A_{2}+\ldots +\left(
-1\right) ^{n}A_{n}
\]

Notice that the coefficients $A_{n}$ in (\ref{AA}) can be rewritten as the
multiple integrals 
\begin{eqnarray*}
A_{1} &=&\int_{0}^{\tau -\delta }\int_{0}^{t_{1}}q\left( t_{2}\right)
dt_{2}dt_{1}+\int_{0}^{\tau -\delta }\int_{t_{1}}^{\tau -\delta }q\left(
t_{2}\right) dt_{2}dt_{1} \\
A_{2} &=&\int_{\delta }^{\tau -\delta }\int_{0}^{t_{1}}q\left( t_{2}\right)
dt_{2}\int_{t_{1}}^{\tau -\delta }\int_{t_{1}}^{t_{3}}q\left( t_{4}\right)
dt_{4}dt_{3}dt_{1}+\int_{\delta }^{\tau -\delta }\int_{t_{1}}^{\tau -\delta
}q\left( t_{2}\right) dt_{2}\int_{0}^{t_{1}}\int_{t_{3}}^{t_{1}}q\left(
t_{4}\right) dt_{4}dt_{3}dt_{1} \\
A_{3} &=&\int_{\delta }^{\tau -3\delta }\int_{0}^{t_{1}}q\left( t_{2}\right)
dt_{2}\int_{t_{1}}^{\tau -\delta }\int_{t_{1}}^{t_{3}}q\left( t_{4}\right)
dt_{4}\int_{t_{3}}^{\tau -\delta }\int_{t_{3}}^{t_{5}}q\left( t_{6}\right)
dt_{6}dt_{5}dt_{3}dt_{1} \\
&&+\int_{2\delta }^{\tau -\delta }\int_{t_{1}}^{\tau -\delta }q\left(
t_{2}\right) dt_{2}\int_{\delta }^{t_{1}}\int_{t_{3}}^{t_{1}}q\left(
t_{4}\right) dt_{4}\int_{0}^{t_{3}}\int_{t_{5}}^{t_{3}}q\left( t_{6}\right)
dt_{6}dt_{5}dt_{3}dt_{1} \\
&&\vdots \\
A_{n} &=&\int_{\delta }^{\tau -n\delta }\int_{0}^{t_{1}}q\left( t_{2}\right)
\ldots \int_{t_{2n-5}}^{\tau -\delta }\int_{t_{2n-5}}^{t_{2\bar{n}%
-3}}q\left( t_{2n-2}\right) \int_{t_{2n-3}}^{\tau -\delta
}\int_{t_{2n-3}}^{t_{2n-1}}q\left( t_{2n}\right)
dt_{2n}dt_{2n-1}\ldots dt_2 dt_{1} \\
&&+\int_{2\delta }^{\tau -\delta }\int_{t_{1}}^{\tau -\delta }q\left(
t_{2}\right) \ldots \int_{\delta
}^{t_{2n-5}}\int_{t_{2n-3}}^{t_{2n-5}}q\left( t_{_{2n-2}}\right)
\int_{0}^{t_{2n-3}}\int_{t_{2n-1}}^{t_{2n-3}}q\left(
t_{2n}\right) dt_{2n}dt_{2n-1}\ldots dt_2 dt_{1}
\end{eqnarray*}%
since $2^{k}\rightarrow \infty $ then, $\delta \rightarrow 0$ and for values
of $\bar{n}$ such that $\bar{n}\delta \rightarrow 0$ then 
\begin{eqnarray}
A_{1} &=&\int_{0}^{\tau }\int_{0}^{t_{1}}q\left( t_{2}\right)
dt_{2}dt_{1}+\int_{0}^{\tau }\int_{t_{1}}^{\tau }q\left( t_{2}\right)
dt_{2}dt_{1}  \nonumber \\
A_{2} &=&\int_{0}^{\tau }\int_{0}^{t_{1}}\int_{t_{1}}^{\tau
}\int_{t_{1}}^{t_{3}}q\left( t_{2}\right) q\left( t_{4}\right)
dt_{4}dt_{3}dt_{2}dt_{1}+\int_{0}^{\tau }\int_{t_{1}}^{\tau
}\int_{0}^{t_{1}}\int_{t_{3}}^{t_{1}}q\left( t_{2}\right) q\left(
t_{4}\right) dt_{4}dt_{3}dt_{2}dt_{1}  \nonumber \\
A_{3} &=&\int_{0}^{\tau }\int_{0}^{t_{1}}\int_{t_{1}}^{\tau
}\int_{t_{1}}^{t_{3}}\int_{t_{3}}^{\tau }\int_{t_{3}}^{t_{5}}q\left(
t_{2}\right) q\left( t_{4}\right) q\left( t_{6}\right)
dt_{6}dt_{5}dt_{4}dt_{3}dt_{2}dt_{1}  \label{re} \\
&&+\int_{0}^{\tau }\int_{t_{1}}^{\tau
}\int_{0}^{t_{1}}\int_{t_{3}}^{t_{1}}\int_{0}^{t_{3}}\int_{t_{5}}^{t_{3}}q%
\left( t_{2}\right) q\left( t_{4}\right) q\left( t_{6}\right)
dt_{6}dt_{5}dt_{4}dt_{3}dt_{2}dt_{1}  \nonumber \\
&&\vdots  \nonumber \\
A_{\bar{n}} &=&\int_{0}^{\tau }\int_{0}^{t_{1}}\ldots \int_{t_{2\bar{n}%
-5}}^{\tau }\int_{t_{2\bar{n}-5}}^{t_{2\bar{n}-3}}\int_{t_{2\bar{n}%
-3}}^{\tau }\int_{t_{2\bar{n}-3}}^{t_{2\bar{n}-1}}p\left( t_{2}\right)
\ldots p\left( t_{2\bar{n}-2}\right) p\left( t_{2\bar{n}}\right) dt_{2\bar{n}%
}dt_{2\bar{n}-1}\ldots dt_{2}dt_{1}  \nonumber
\\
&&+\int_{0}^{\tau }\int_{t_{1}}^{\tau }\ldots \int_{0}^{t_{2\bar{n}%
-5}}\int_{t_{2\bar{n}-3}}^{t_{2\bar{n}-5}}\int_{0}^{t_{2\bar{n}-3}}\int_{t_{2%
\bar{n}-1}}^{t_{2\bar{n}-3}}p\left( t_{2}\right) \ldots p\left( t_{2\bar{n}%
-2}\right) p\left( t_{2\bar{n}}\right) dt_{2\bar{n}}dt_{2\bar{n}-1}\ldots dt_{2}dt_{1}  \nonumber
\end{eqnarray}%
It remains to be proven that the coefficients $A_{n}$ are equal to the
coefficients obtained by Lyapunov, which is simply done by rearranging the
integration variables, see Lemma in \ref{lemmaA42} for the first three
coefficients, 
\begin{eqnarray*}
A_{1} &=&\tau \int_{0}^{\tau }q\left( t_{2}\right) dt_{2} \\
A_{2} &=&\int_{0}^{\tau }\int_{0}^{t_{2}}\left( \tau -t_{2}+t_{1}\right)
\left( t_{2}-t_{1}\right) q\left( t_{1}\right) q\left( t_{2}\right)
dt_{1}dt_{2} \\
A_{3} &=&\int_{0}^{\tau }\int_{0}^{t_{6}}\int_{0}^{t_{4}}\left( \tau
-t_{6}+t_{2}\right) \left( t_{6}-t_{4}\right) \left( t_{4}-t_{2}\right)
q\left( t_{2}\right) q\left( t_{4}\right) q\left( t_{6}\right)
dt_{2}dt_{4}dt_{6}
\end{eqnarray*}%
the general term is 
\[
A_{n}=\int_{0}^{\tau }dt_{1}\int_{0}^{t_{1}}dt_{2}\ldots \int_{0}^{t_{\bar{n}%
-1}}\left( \tau -t_{1}+t_{n}\right) \left( t_{1}-t_{2}\right) \ldots \left(
t_{n-1}-t_{n}\right) q\left( t_{1}\right) q\left( t_{2}\right) \ldots
q\left( t_{n}\right) dt_{n} 
\]%
thus, the theorem follows.
\end{proof}

Notice that the coefficients $A_{n}$ of Theorem \ref{theorem13} are equal to the
coefficients of the Lyapunov approximation, see section 1, but for a factor $%
\frac{1}{2}$, this difference is because we consider the discriminant as $%
\Delta \left( \alpha ,\beta \right) =x_{1}\left( \tau \right) +\dot{x}%
_{2}\left( \tau \right) $ and Lyapunov defined its characteristic constant
as $A=\frac{1}{2}\left( x_{1}\left( \tau \right) +\dot{x}_{2}\left( \tau
\right) \right) $.

\begin{remark} \label{remark14}
Since the approximation made by Lyapunov depends on multiple integrals (the
number of multiple integrals that one has to calculate is equal to the
sub index of each coefficient $A_{n}$), it is very hard to compute.
Nevertheless, Theorem \ref{theorem11} give us a recursive method for obtaining the
approximation of $\Delta \left( \alpha ,\beta \right) $.
\end{remark}

By doing $2^{k}\rightarrow \infty $, we have taken the approximation of
Theorem \ref{theorem11}, which depends on some discrete values of the function $p\left(
t\right) $ and we have transformed it to an approximation that depends on
definite integrals of the excitation function $p\left( t\right) $. So, one
can say that Theorem \ref{theorem11} may be seen as a "discrete" form of the discriminant
approximation made by Lyapunov.

\section{Numerical calculation of the discriminant approximation $\Delta \left( \protect\alpha ,\protect\beta \right) $}

As we know the discriminant $\Delta \left( \alpha ,\beta \right) $ plays a
very important role in the determination of the stability zones of linear
periodic differential equations. By using Theorem \ref{theorem11} we are able to compute
an approximation of the discriminant $\Delta \left( \alpha ,\beta \right) $
at each point of the $\alpha -\beta $ plane and then use the stability
conditions of Theorem \ref{theorem4} to find the stable zones ($\left\vert \Delta \left(
\alpha ,\beta \right) \right\vert <2$), unstable zones ($\left\vert \Delta
\left( \alpha ,\beta \right) \right\vert >2$) or the transition curves ($%
\left\vert \Delta \left( \alpha ,\beta \right) \right\vert =2$) of any Hill
equation.

We must notice that the discriminant $\Delta \left( \alpha ,\beta \right) $
defines a manifold $\left( \alpha ,\beta ,\Delta \left( \alpha ,\beta
\right) \right) $ in $%
\mathbb{R}
^{3}$ which contains all the stability properties of the periodic
differential equation. It is not so hard to see that the projection of the
intersection between the manifold $\left( \alpha ,\beta ,\Delta \left(
\alpha ,\beta \right) \right) $ and the surfaces $surf_{1}=\left\{ \left(
\alpha ,\beta ,z\right) |\forall \alpha ,\beta \in 
\mathbb{R}
\text{, }z=2\right\} $ and $surf_{2}=\left\{ \left( \alpha ,\beta ,z\right)
|\forall \alpha ,\beta \in 
\mathbb{R}
\text{, }z=-2\right\} $, in the $\alpha -\beta $ plane, are the transition
curves of a Hill equation, see Fig. 2.%

   \begin{figure}[h]
      \centering
      \includegraphics[trim=20mm 0mm 20mm 25mm,clip,height=1.8in,width=0.5\textwidth]{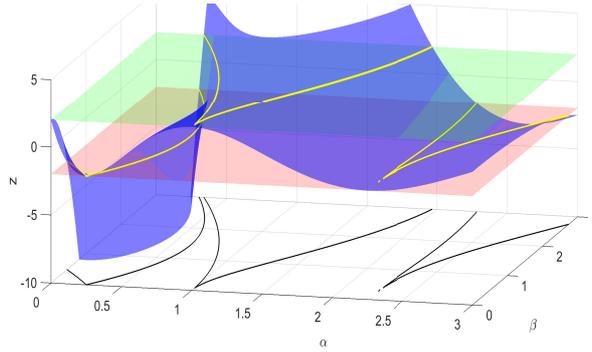}
      \caption{Discriminant approximation manifold in blue, $surf_{1}$ and $surf_{2}$ in red and green respectively, yellow lines represents the intersection between $\Delta\left( \protect\alpha ,\protect\beta \right) $ and $surf_{1}$ or $surf_{2}$, transition curves in black.}
      
   \end{figure}


Figure 3 shows the approximation, for two different approximation orders,
and the actual transition curves of the periodic differential equation 
\begin{equation}
\ddot{x}+(\alpha +\beta \left( \cos \left( t\right) +\cos \left( 2t\right)
\right) )x=0  \label{L}
\end{equation}%
\begin{figure}[h]
\centering
\subfigure[]{\includegraphics[trim=40mm 5mm 40mm 10mm,clip,height=1.4in,width=0.43\textwidth]{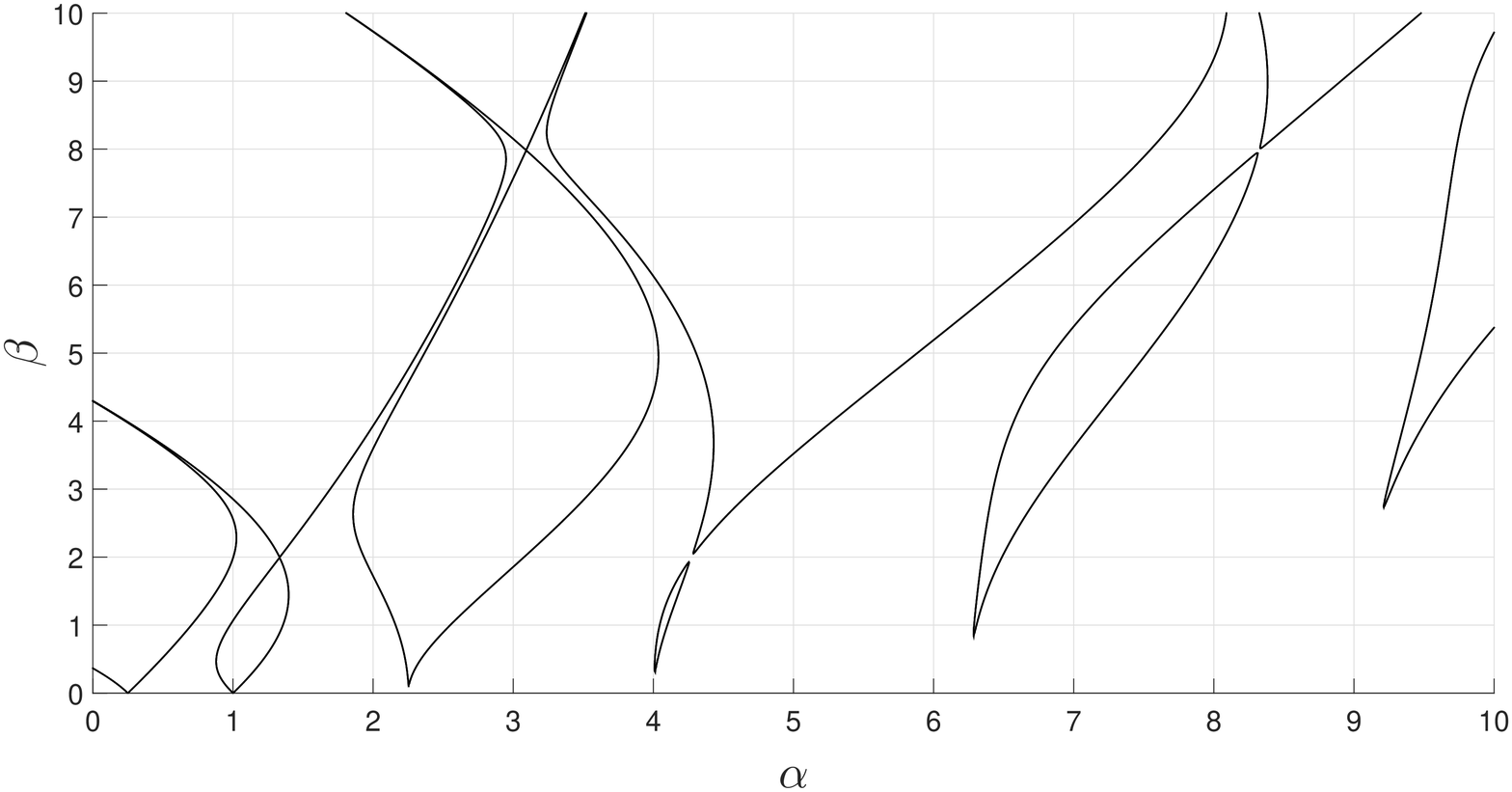}}
\subfigure[]{\includegraphics[trim=30mm 5mm 40mm 10mm,clip,height=1.4in,width=0.43\textwidth]{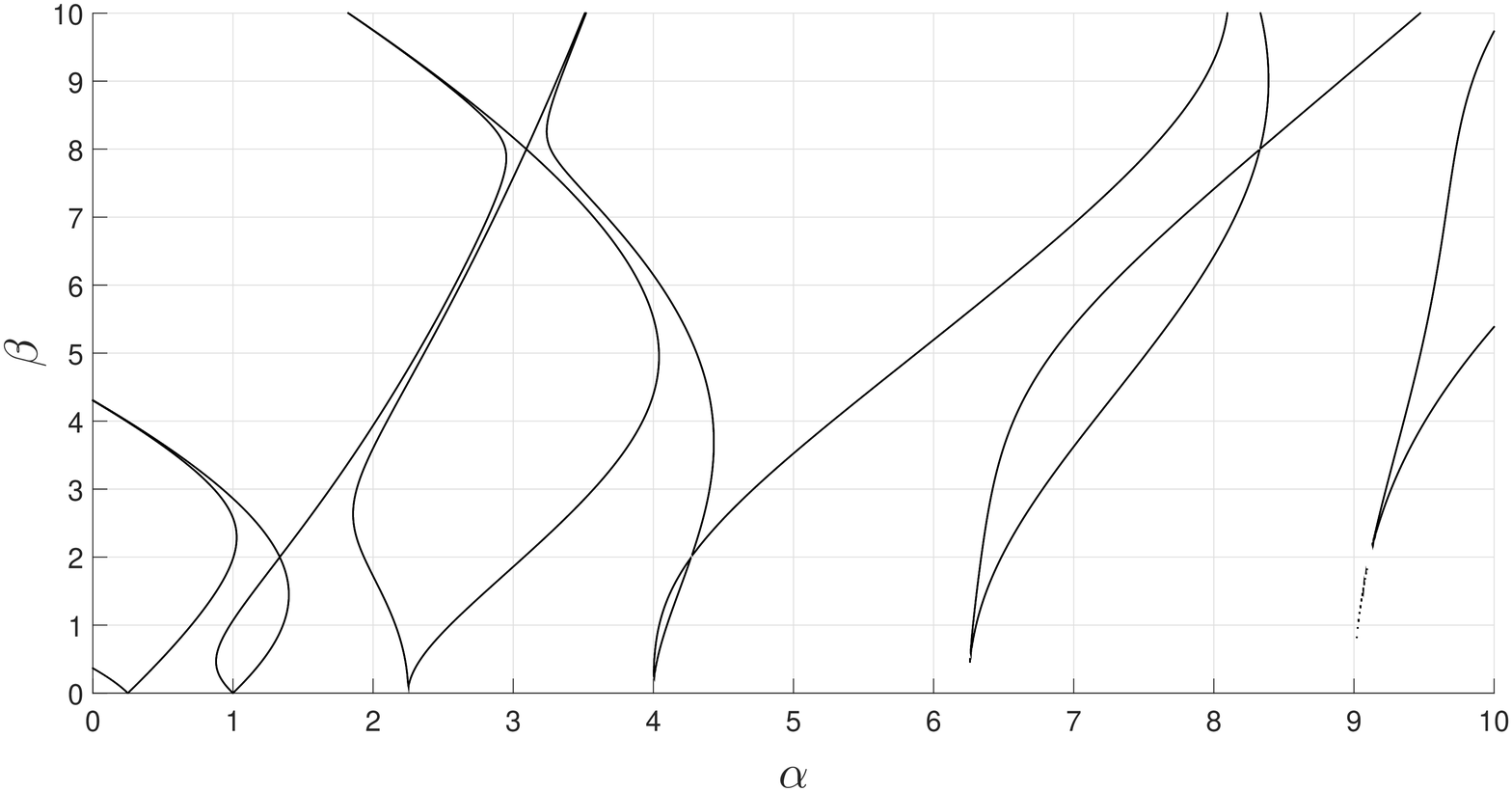}}
 \subfigure[]{\includegraphics[trim=30mm 5mm 40mm 10mm,clip,height=1.4in,width=0.43\textwidth]{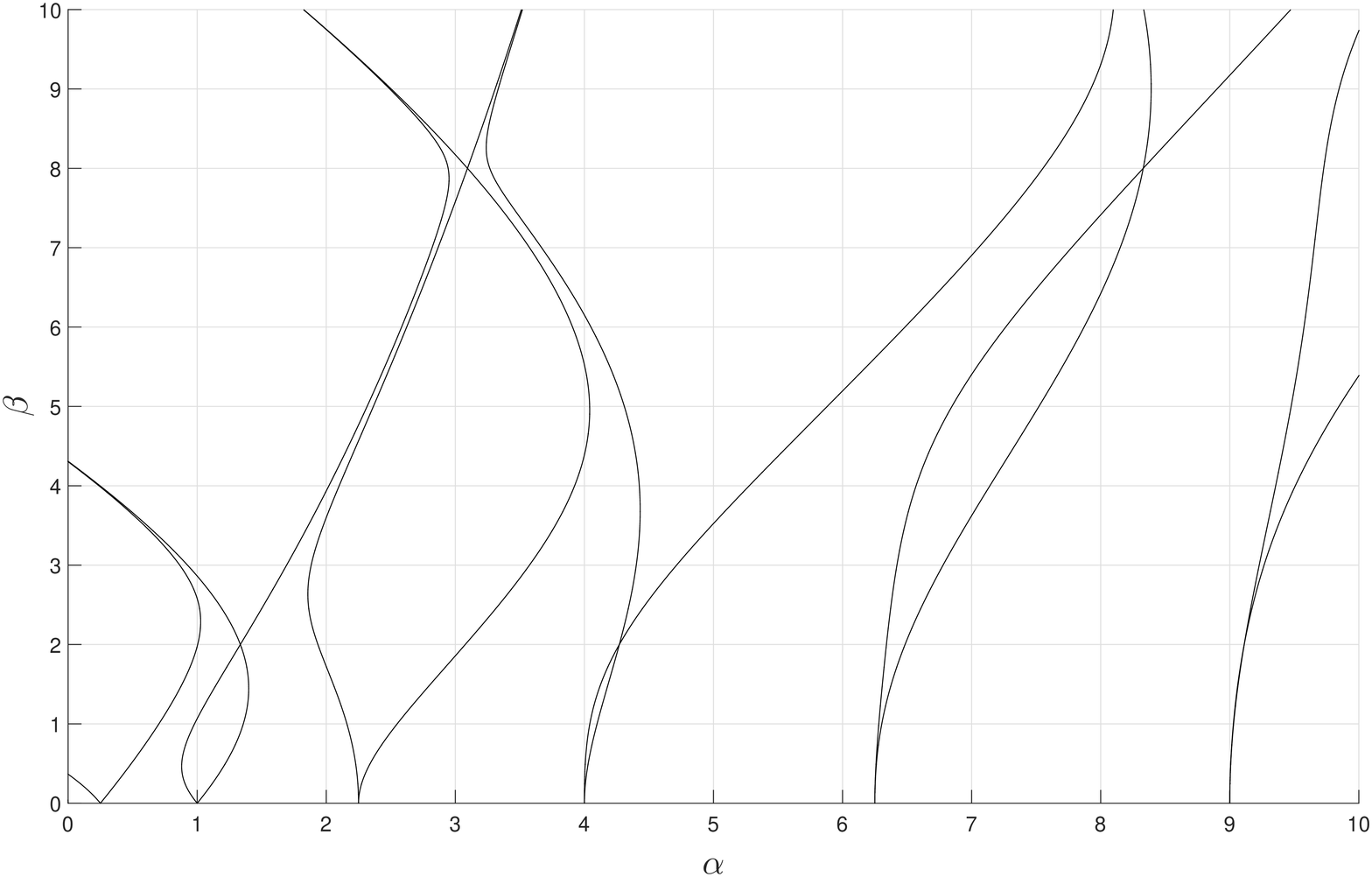}}
\caption{a) Transition curves approximation of order $2^{10}$, b) Transition curves approximation of order $2^{15}$ c) Actual transition curves of (\ref{L}).}
\end{figure}

Notice that, as it was expected, the approximation of the transition curves
of order $2^{15}$ is better than the one with order $2^{10}$, nevertheless, the
latter gives us an idea of the shape of the actual transition
curves.

As a final remark we must say that the approximation here developed, was
obtained thanks to the Walsh function properties. It is true that all
orthogonal series, such as Fourier series, Bessel functions, share some
common properties such as the possibility of obtaining integral or
differential operators or the closeness under the multiplication, that is,
if $f\left( t\right) $ and $g\left( t\right) $ belong to an infinite set of
orthogonal series the product $f\left( t\right) g\left( t\right) $ may be
represented in terms of the infinite set of orthogonal series. But, Walsh
functions have some extra properties, that just a few other sets of orthogonal
series have, for example the multiplication of two Walsh functions belonging
to a finite set may be expressed in terms of the finite set elements.
Another property is the one presented in Lemma \ref{lemma6}, see section 3, which
allows us to post factorize the vector of Walsh function $\bar{w}_{2^{k}}$
given the multiplication $\bar{w}_{2^{k}}\bar{w}_{2^{k}}^{T}\gamma $; this
property made it possible to obtain the state transition matrix
approximation in (\ref{STM}). And, finally the almost orthogonal similarity
of the upper triangular matrices: $\bar{\Lambda}_{\bar{r}}=W_{H}^{-1}\Lambda
_{\bar{r}}W_{H}$ ($p\left( t\right) $ sampling matrix), $\bar{P}=W_{H}PW_{H}$
(integration operator), $\bar{\Gamma}=W_{H}^{-1}\bar{\Gamma}W_{H}$ and $\bar{%
\Gamma}_{P}=W_{H}^{-1}\bar{\Gamma}_{P}W_{H}$ (discriminant sampling
matrices), which is the most fundamental property for the development of the
approximation (\ref{summ2}), see Theorem \ref{theorem11}.

\section{Conclusion}

In section 5 we have given an alternative proof of the discriminant
approximation $\Delta \left( \alpha ,\beta \right) $ made by Lyapunov in 
\cite{Lyapunov} and studied in depth in \cite{LyapunovApp}. This new proof
is based on some properties of Walsh functions and basic definitions and
properties of multiple integrals. In spite of being one of the most known
approximations, the Lyapunov discriminant approximation is very difficult to
obtain since it is necessary to calculate a very large number of multiple
integrals. In this work, the alternative is easily programmed and allows us
to have a computational approach to the discriminant approximation $\Delta
\left( \alpha ,\beta \right) $ to any desired accuracy.

We give a new approximation of $\Delta \left( \alpha ,\beta \right) $, this
is a recursive summation and it only depends on the evaluation of the
excitation function at the time $t_{n}\in \lbrack 0,\tau ]$, $n=0,1,2,\ldots
2^{k}-1$, where $2^k$ is the order of the approximation, see Theorem \ref{theorem11}.
First, the new approximation of $\Delta \left( \alpha ,\beta \right) $ was
obtained by means of Walsh functions, $\Delta \left( \alpha ,\beta \right) $
depended on two large dimensional inverse matrices $\Gamma $ and $\Gamma
_{p} $, see (\ref{GH}), the dependence of $\Delta \left( \alpha ,\beta
\right) $ on Walsh functions was then eliminated by a similarity
transformation of $\Gamma $ and $\Gamma _{p}$, thus the new form of $\Delta
\left( \alpha ,\beta \right) $ depended on the summation of the last column
entries of the two large dimensional triangular inverse matrices $\bar{\Gamma%
}$ and $\bar{\Gamma}_{p}$, see equation (\ref{Dapp2}) on Lemma \ref{lemma8}. Then, the
last column entries of $\bar{\Gamma}$ and $\bar{\Gamma}_{p}$ were calculated
and the recursive summation was obtained.

The new approximation, Theorem \ref{theorem11}, may be seen as a "discrete" form of the
discriminant approximation made by Lyapunov. This approximation is easy to
implement on a computer. The latter is important
because the accuracy of the approximation depends on the number of
recursions, i.e. as $2^{k}\rightarrow \infty $ the accuracy of the method
will be better.
\appendix

\section{Walsh function vector permutation matrix}
Next Lemma states that the columns of the symmetric matrix defined as $%
M_{2^{k}}=\bar{w}_{2^{k}}\bar{w}_{2^{k}}^{T}$ are the permutations of the
vector $\bar{w}_{2^{k}}$. The proof is based on the fact that if $%
w_{n}\left( t\right) $ and $w_{m}\left( t\right) $ belong to a finite set of
Walsh functions $\left\{ w_{0}\left( t\right) ,w_{1}\left( t\right) ,\ldots
,w_{2^{k}-1}\left( t\right) \right\} $ then, the multiplication $w_{n}\left(
t\right) w_{m}\left( t\right) $ belongs to the same finite set to which the
functions $w_{n}\left( t\right) $ and $w_{m}\left( t\right) $ belong. 

\begin{lemma} \label{lemmaA1}
Let $\bar{w}_{2^{k}}\left( t\right) $ be the $2^{k}\times 1$ vector of first 
$2^{k}$ Walsh functions $w_{n}\left( t\right) $, $n=0,1,\ldots ,2^{k}-1$,
and $M_{2^{k}}=\bar{w}_{2^{k}}\bar{w}_{2^{k}}^{T}$ then the columns of the
matrix $M$ are permutations of the entries of the vector $\bar{w}%
_{2^{k}}\left( t\right) $ and it can be written as 
\[
M_{2^{k}}\left( t\right) =\left[ \bar{w}_{2^{k}}\left( t\right) ,\Lambda
_{1}^{\left( 2^{k}\right) }\bar{w}_{2^{k}}\left( t\right) ,...,\Lambda
_{2^{k}-1}^{\left( 2^{k}\right) }\bar{w}_{2^{k}}\left( t\right) \right] 
\]%
where 
\begin{eqnarray*}
\Lambda _{i}^{\left( 2^{k}\right) } &=&\left[ 
\begin{array}{cc}
\Lambda _{i}^{\left( 2^{k}/2\right) } & 0_{\left( 2^{k}/2\right) } \\ 
0_{\left( 2^{k}/2\right) } & \Lambda _{i}^{\left( 2^{k}/2\right) }%
\end{array}%
\right] \\
\Lambda _{i+\left( 2^{k}/2\right) }^{\left( 2^{k}\right) } &=&\left[ 
\begin{array}{cc}
0_{\left( 2^{k}/2\right) } & \Lambda _{i}^{\left( 2^{k}/2\right) } \\ 
\Lambda _{i}^{\left( 2^{k}/2\right) } & 0_{\left( 2^{k}/2\right) }%
\end{array}%
\right] \\
\Lambda _{0}^{\left( 2^{k}\right) } &=&I_{2^{k}} \\
\Lambda _{i}^{\left( 2^{k}\right) } &\in &R^{2^{k}\times 2^{k}}
\end{eqnarray*}
\end{lemma}

The proof may be found in \cite{Stavroulakis}. For didactic purposes,
let us take $k=2$ and since $w_{n}\left( t\right) w_{m}\left( t\right)
=w_{n\oplus m}\left( t\right) $, where $\oplus $ refers the the no-carry
modulo-2 addition, the matrix $M_{4}$ is

\[
M_{4}\left( t\right) =\left[ 
\begin{array}{cccc}
w_{0}\left( t\right) & w_{1}\left( t\right) & w_{2}\left( t\right) & 
w_{3}\left( t\right) \\ 
w_{1}\left( t\right) & w_{0}\left( t\right) & w_{3}\left( t\right) & 
w_{2}\left( t\right) \\ 
w_{2}\left( t\right) & w_{3}\left( t\right) & w_{0}\left( t\right) & 
w_{1}\left( t\right) \\ 
w_{3}\left( t\right) & w_{2}\left( t\right) & w_{1}\left( t\right) & 
w_{0}\left( t\right)%
\end{array}%
\right] 
\]%
from where it can be seen that the second, third and forth columns of $M_{4}$
are permutations of the entries of the first column, even more the second
column can be rewritten as%
\[
\left[ 
\begin{array}{c}
w_{1}\left( t\right) \\ 
w_{0}\left( t\right) \\ 
w_{3}\left( t\right) \\ 
w_{2}\left( t\right)%
\end{array}%
\right] =\left[ 
\begin{array}{cccc}
0 & 1 & 0 & 0 \\ 
1 & 0 & 0 & 0 \\ 
0 & 0 & 0 & 1 \\ 
0 & 0 & 1 & 0%
\end{array}%
\right] \left[ 
\begin{array}{c}
w_{0}\left( t\right) \\ 
w_{1}\left( t\right) \\ 
w_{2}\left( t\right) \\ 
w_{3}\left( t\right)%
\end{array}%
\right] 
\]%
or simply $\Lambda _{1}\bar{w}_{4}\left( t\right) $, the third column is 
\[
\left[ 
\begin{array}{c}
w_{2}\left( t\right) \\ 
w_{3}\left( t\right) \\ 
w_{0}\left( t\right) \\ 
w_{1}\left( t\right)%
\end{array}%
\right] =\left[ 
\begin{array}{cccc}
0 & 0 & 1 & 0 \\ 
0 & 0 & 0 & 1 \\ 
1 & 0 & 0 & 0 \\ 
0 & 1 & 0 & 0%
\end{array}%
\right] \left[ 
\begin{array}{c}
w_{0}\left( t\right) \\ 
w_{1}\left( t\right) \\ 
w_{2}\left( t\right) \\ 
w_{3}\left( t\right)%
\end{array}%
\right] 
\]%
or simply $\Lambda _{2}\bar{w}_{4}\left( t\right) $, and the fourth column is%
\[
\left[ 
\begin{array}{c}
w_{3}\left( t\right) \\ 
w_{2}\left( t\right) \\ 
w_{1}\left( t\right) \\ 
w_{0}\left( t\right)%
\end{array}%
\right] =\left[ 
\begin{array}{cccc}
0 & 0 & 0 & 1 \\ 
0 & 0 & 1 & 0 \\ 
0 & 1 & 0 & 0 \\ 
1 & 0 & 0 & 0%
\end{array}%
\right] \left[ 
\begin{array}{c}
w_{0}\left( t\right) \\ 
w_{1}\left( t\right) \\ 
w_{2}\left( t\right) \\ 
w_{3}\left( t\right)%
\end{array}%
\right] 
\]%
or simply $\Lambda _{3}\bar{w}_{4}\left( t\right) $, so 
\[
M_{4}\left( t\right) =\left[ \bar{w}_{4}\left( t\right) ,\Lambda _{1}\bar{w}%
_{4}\left( t\right) ,\Lambda _{2}\bar{w}_{4}\left( t\right) ,\Lambda _{3}%
\bar{w}_{4}\left( t\right) \right] 
\]%
which is consistent with the above mentioned lemma.

\section{Obtaining the last column of an upper triangular inverse matrix }
In this part, the last column entries of a non-singular upper triangular
matrix, are obtained as a recursive summation.

Let $U$ be the $n\times n$ real non-singular matrix 
\begin{equation}
U=\left[ 
\begin{array}{cccccc}
u_{1,1} & u_{1,2} & \cdots  & u_{1,n-2} & u_{1,n-1} & u_{1,n} \\ 
0 & u_{22} & \cdots  & u_{2,n-2} & u_{2,n-1} & u_{2,n} \\ 
\vdots  & \vdots  & \ddots  & \vdots  & \vdots  & \vdots  \\ 
0 & 0 & \cdots  & u_{n-2,n-2} & u_{n-2,n-1} & u_{n-2,n} \\ 
0 & 0 & \cdots  & 0 & u_{n-1,n-1} & u_{n-1,n} \\ 
0 & 0 & \cdots  & 0 & 0 & u_{n,n}%
\end{array}%
\right]   \label{U}
\end{equation}%
and define its inverse as 
\begin{equation}
U^{-1}=\left[ 
\begin{array}{cccccc}
\ast  & \ast  & \cdots  & \ast  & \ast  & a_{n-1} \\ 
0 & \ast  & \cdots  & \ast  & \ast  & a_{n-2} \\ 
\vdots  & \vdots  & \ddots  & \vdots  & \vdots  & \vdots  \\ 
0 & 0 & \cdots  & \ast  & \ast  & a_{2} \\ 
0 & 0 & \cdots  & 0 & \ast  & a_{1} \\ 
0 & 0 & \cdots  & 0 & 0 & a_{0}%
\end{array}%
\right]   \label{Ui}
\end{equation}%
the last column entries of the inverse of $U$, i.e. the coefficients $%
a_{\ell }$, $\ell =0,1,2,\ldots ,n-1$, may be calculated as follows

\begin{lemma} \label{lemmaA2}
Let $U$ be defined as in (\ref{U}), then the last column entries of $U^{-1}$
are%
\begin{eqnarray}
a_{0} &=&\frac{1}{u_{n,n}} \\
a_{k} &=&\sum_{j=0}^{k-1}-a_{j}\frac{u_{n-k,n-j}}{u_{n-k,n-k}}  \nonumber \\
k &=&1,2,\ldots n  \nonumber
\end{eqnarray}%
the coefficients $a_{k}$ are defined as in (\ref{Ui}).
\end{lemma}

\begin{proof}
It is known that the inverse of any non-singular matrix $R$ is $R^{-1}=\frac{%
1}{\det \left( R\right) }adj\left( R\right) $, the non-singularity of $U$
guarantees that $\det \left( U\right) \not=0$ moreover $\det \left( U\right)
=\prod_{i=1}^{n}u_{i,i}$. By direct calculation, the $n,n$ entry of $U^{-1}$
is%
\begin{eqnarray*}
a_{0} &\triangleq &\frac{1}{\det \left( U\right) }adj\left( U\right) _{n,n}=%
\frac{1}{u_{n,n}} \\
a_{1} &\triangleq &\frac{1}{\det \left( U\right) }adj\left( U\right)
_{n-1,n}=\frac{u_{n-1,n}}{u_{n,n}u_{n-1,n-1}}=-a_{0}\frac{u_{n-1,n}}{%
u_{n-1,n-1}} \\
a_{2} &\triangleq &\frac{1}{\det \left( U\right) }adj\left( U\right)
_{n-2,n}=\frac{u_{n-2,n-1}u_{n-1,n}-u_{n-2,n}u_{n-1,n-1}}{%
u_{n-2,n-2}u_{n-1,n-1}u_{n,n}}=-a_{1}\frac{u_{n-2,n-1}}{u_{n-2,n-2}}-a_{0}%
\frac{u_{n-2,n}}{u_{n-2,n-2}} \\
&&\vdots \\
a_{i} &\triangleq &\frac{1}{\det \left( U\right) }adj\left( U\right)
_{n-i,n}=-a_{i-1}\frac{u_{n-i,n-i-1}}{u_{n-i,n-i}}-a_{i-2}\frac{u_{n-i,n-i-2}%
}{u_{n-i,n-i}}-\cdots -a_{1}\frac{u_{n-i,n-1}}{u_{n-i,n-i}}-a_{0}\frac{%
u_{n-i,n}}{u_{n-i,n-i}}
\end{eqnarray*}%
and the lemma follows.
\end{proof}

\section{Proof of the second part of the Theorem \ref{theorem11}} 

In order to obtain the coefficients $c_{n}$ of Theorem \ref{theorem11} we must remember
that each coefficient $c_{n}$, $n=0,1,...2^{k}-1$, is the $2^{k}-n$ entry of
the last column of $\bar{\Gamma}_{P}\left( \tau \right) $ and doing almost
the same as for coefficients $b_{n}$ we have%
\begin{eqnarray*}
c_{0} &=&\frac{2^{2k+2}}{2^{2k+2}+\tau ^{2}\left( \alpha +\beta
p_{2^{k}}\right) } \\
c_{1} &=&-4\tau ^{2}\frac{\alpha +\frac{\beta }{2}\left(
p_{2^{k}-1}+p_{2^{k}}\right) }{2^{2k+2}+\tau ^{2}\left( \alpha +\beta
p_{2^{k}-1}\right) }c_{0} \\
c_{2} &=&4\tau ^{2}\frac{-c_{0}\left( 2\alpha +\beta \left( \frac{1}{2}%
p_{2^{k}-2}+p_{2^{k}-1}+\frac{p_{2^{k}}}{2}\right) \right) -c_{1}\left(
\alpha +\frac{\beta }{2}\left( p_{2^{k}-2}+p_{2^{k}-1}\right) \right) }{%
2^{2k+2}+\tau ^{2}\left( \alpha +\beta p_{2^{k}-2}\right) } \\
&&\vdots  \\
c_{n} &=&\frac{4\tau ^{2}}{2^{2k+2}+\tau ^{2}\left( \alpha +\beta
p_{2^{k}-n}\right) }\left( -\sum\limits_{i=0}^{n-1}\left( c_{i}\left( \left(
n-i\right) \alpha -\frac{1}{2}\left( p_{2^{k}-n}+p_{2^{k}-i}\right)
+\sum_{j=i}^{n}p_{2^{k}-j}\right) \right) \right) 
\end{eqnarray*}%
if we define 
\[
\mu _{h}=\alpha +\frac{\beta }{2}\left( p_{2^{k}-h}+p_{2^{k}-h+1}\right) 
\]%
then the coefficients $c_{n}$ become%
\begin{eqnarray*}
c_{0} &=&\frac{2^{2k+2}}{2^{2k+2}+\tau ^{2}\left( \alpha +\beta
p_{2^{k}}\right) } \\
c_{n} &=&-\psi _{n}\sum\limits_{i=0}^{n-1}\left( c_{i}\sum_{j=i+1}^{n}\mu
_{j}\right) 
\end{eqnarray*}%
for $n=1,2,\ldots 2^{k}-1$, $\psi _{h}$ is defined as in (\ref{Psih}).

\section{} 

In this part two Lemmas are stated: the first one says that the recurrent
summations (\ref{A41}) can be written in terms of powers of the real
constant $\delta $; and the second one gives us a rewriting of the multiple
integrals shown in (\ref{re1}). The proof of the former is done by following
the formulas (\ref{A41}), and the proof of the latter is done by rearranging
the limits and integration variables. Both Lemmas are used in the proof of
Theorem \ref{theorem13}. 

\begin{lemma} \label{lemmaA41}
Let 
\begin{eqnarray}
A_{0} &=&B_{0}=1  \nonumber \\
A_{n} &=&1-\delta \sum_{i=0}^{n-1}\left[ A_{i}\mathbb{I}_{n+1,i+1}\right] 
\label{A41} \\
B_{n} &=&1-\delta \left[ \sum_{i=0}^{n-1}\left( n-i\right) B_{i}\mathbb{I}%
_{i+2,i+1}\right]   \nonumber
\end{eqnarray}%
where $\mathbb{I}_{j,\ell }$ and $\delta $ are real constants, then $A_{n}$
and $B_{n}$ can be rewritten in terms of powers of $\delta $ as 
\begin{eqnarray*}
A_{n} &=&1-\delta \sum_{i=1}^{n}\mathbb{I}_{n+1,i}+\delta ^{2}\sum_{i=2}^{n}%
\mathbb{I}_{n+1,i}\sum_{j=1}^{i-1}\mathbb{I}_{i,j}-\delta ^{3}\sum_{i=3}^{n}%
\mathbb{I}_{n+1,i}\sum_{j=1}^{i-1}\mathbb{I}_{i,j}\sum_{l=1}^{j}\mathbb{I}%
_{j,l}+\ldots  \\
&&%
\begin{tabular}{l}
$\ \ \ \ \ \ \ \ \ \ \ \ \ \ \ \ \ \ \ \ \ \ \ \ \ \ \ \ \ \ \ \ \ \ \ \ \ \
\ \ \  \ \ \ \ldots +\left( -1\right) ^{\bar{n}}\delta ^{%
\bar{n}}\sum_{i=\bar{n}}^{n}\mathbb{I}_{n+1,i}\sum_{j=1}^{i-1}\mathbb{I}%
_{i,j}\sum_{l=1}^{j}\mathbb{I}_{j,l}\ldots \sum_{y=1}^{x}\mathbb{I}_{x,y}$%
\end{tabular}
\\
B_{n} &=&1-\delta \sum_{i=2}^{n+1}\mathbb{I}_{i,1}+\delta ^{2}\sum_{i=2}^{n}%
\mathbb{I}_{i,1}\sum_{j=i+1}^{n+1}\mathbb{I}_{j,i}-\delta
^{3}\sum_{i=2}^{n-1}\mathbb{I}_{i,1}\sum_{j=i+1}^{n}\mathbb{I}%
_{j,i}\sum_{l=j+1}^{n+1}\mathbb{I}_{l,j}+\ldots  \\
&&%
\begin{tabular}{l}
$\ \ \ \ \ \ \ \ \ \ \ \ \ \ \ \ \ \ \ \ \ \ \ \ \ \ \ \ \ \ \ \ \ \ \ \ \ \
\ \ \ \ \ \ \ \ \ \ \ldots +\left( -1\right) ^{\bar{%
n}}\delta ^{\bar{n}}\sum_{i=2}^{n+2-\bar{n}}\mathbb{I}_{i,1}%
\sum_{j=i+1}^{n+3-\bar{n}}\mathbb{I}_{j,i}\ldots \sum_{y=x+1}^{n+1}\mathbb{I}%
_{y,x}$%
\end{tabular}%
\end{eqnarray*}
\end{lemma}

\begin{proof}
Following the formula for $S_{n}$ and grouping terms of powers of $\delta $
one obtains%
\begin{eqnarray*}
A_{1} &=&1-\delta \mathbb{I}_{2,1} \\
A_{2} &=&1-\delta \left( \mathbb{I}_{3,1}+A_{1}\mathbb{I}_{3,2}\right)
=1-\delta \left( \mathbb{I}_{3,1}+\mathbb{I}_{3,2}\right) +\delta ^{2}%
\mathbb{I}_{2,1}\mathbb{I}_{3,2} \\
A_{3} &=&1-\delta \left( \mathbb{I}_{4,1}+\mathbb{I}_{4,2}+\mathbb{I}%
_{4,3}\right) +\delta ^{2}\left( \mathbb{I}_{2,1}\mathbb{I}_{4,2}+\left( 
\mathbb{I}_{3,1}+\mathbb{I}_{3,2}\right) \mathbb{I}_{4,3}\right) -\delta ^{3}%
\mathbb{I}_{2,1}\mathbb{I}_{3,2}\mathbb{I}_{4,3} \\
A_{4} &=&1-\delta \left( \mathbb{I}_{5,1}+A_{1}\mathbb{I}_{5,2}+A_{2}\mathbb{%
I}_{5,3}+A_{3}\mathbb{I}_{5,4}\right) \\
&=&1-\delta \left( \mathbb{I}_{5,1}+\mathbb{I}_{5,2}+\mathbb{I}_{5,3}+\mathbb{%
I}_{5,4}\right) +\delta ^{2}\left( \mathbb{I}_{2,1}\mathbb{I}_{5,2}+\left( 
\mathbb{I}_{3,1}+\mathbb{I}_{3,2}\right) \mathbb{I}_{5,3}+\left( \mathbb{I}%
_{4,1}+\mathbb{I}_{4,2}+\mathbb{I}_{4,3}\right) \mathbb{I}_{5,4}\right) - \\
&&-\delta ^{3}\left( \mathbb{I}_{2,1}\mathbb{I}_{3,2}\mathbb{I}_{5,3}+\left( 
\mathbb{I}_{2,1}\mathbb{I}_{4,2}+\left( \mathbb{I}_{3,1}+\mathbb{I}%
_{3,2}\right) \mathbb{I}_{4,3}\right) \mathbb{I}_{5,4}\right) +\delta ^{4}%
\mathbb{I}_{2,1}\mathbb{I}_{3,2}\mathbb{I}_{4,3}\mathbb{I}_{5,4} \\
&&\vdots \\
A_{n}&=&1-\delta \sum_{i=1}^{n}\mathbb{I}_{n+1,i}+\delta ^{2}\sum_{i=2}^{n}%
\mathbb{I}_{n+1,i}\sum_{j=1}^{i-1}\mathbb{I}_{i,j}-\delta ^{3}\sum_{i=3}^{n}%
\mathbb{I}_{n+1,i}\sum_{j=1}^{i-1}\mathbb{I}_{i,j}\sum_{l=1}^{j}\mathbb{I}%
_{j,l}+\ldots \\
&& \qquad\qquad\qquad\qquad\qquad\qquad\qquad\qquad\qquad\qquad \ldots+\left( -1\right) ^{\bar{n}}\delta ^{\bar{n}}\sum_{i=\bar{n}%
}^{n}\mathbb{I}_{n+1,i}\sum_{j=1}^{i-1}\mathbb{I}_{i,j}\sum_{l=1}^{j}\mathbb{%
I}_{j,l}\ldots \sum_{y=1}^{x}\mathbb{I}_{x,y} 
\end{eqnarray*}
doing a similar procedure for $B_{n}$ one gets%
\begin{eqnarray*}
B_{0} &=&1 \\
B_{1} &=&1-\delta \mathbb{I}_{2,1} \\
B_{2} &=&1-\delta \left( \mathbb{I}_{3,1}+\mathbb{I}_{2,1}\right) +\delta
^{2}\mathbb{I}_{2,1}\mathbb{I}_{3,2} \\
B_{3} &=&1-\delta \left( \mathbb{I}_{4,1}+\mathbb{I}_{3,1}+\mathbb{I}%
_{2,1}\right) +\delta ^{2}\left( \mathbb{I}_{2,1}\left( \mathbb{I}_{4,2}+%
\mathbb{I}_{3,2}\right) +\mathbb{I}_{3,1}\mathbb{I}_{4,3}\right) -\delta ^{3}%
\mathbb{I}_{2,1}\mathbb{I}_{3,2}\mathbb{I}_{4,3} \\
&&\vdots \\
B_{n}&=&1-\delta \sum_{i=2}^{n+1}\mathbb{I}_{i,1}+\delta ^{2}\sum_{i=2}^{n}%
\mathbb{I}_{i,1}\sum_{j=i+1}^{n+1}\mathbb{I}_{j,i}-\delta
^{3}\sum_{i=2}^{n-1}\mathbb{I}_{i,1}\sum_{j=i+1}^{n}\mathbb{I}%
_{j,i}\sum_{l=j+1}^{n+1}\mathbb{I}_{l,j}+\ldots \\
&&\qquad\qquad\qquad\qquad\qquad\qquad\qquad\qquad\qquad\qquad\qquad\ldots+\left( -1\right) ^{\bar{n}%
}\delta ^{\bar{n}}\sum_{i=2}^{n+2-\bar{n}}\mathbb{I}_{i,1}\sum_{j=i+1}^{n+3-%
\bar{n}}\mathbb{I}_{j,i}\ldots \sum_{y=x+1}^{n+1}\mathbb{I}_{y,x} 
\end{eqnarray*}
and the lemma follows.
\end{proof}

\bigskip

\begin{lemma} \label{lemmaA42}
Let the multiple definite integral sequence \ 
\begin{eqnarray}
A_{1} &=&\int_{0}^{\tau }\int_{0}^{t_{1}}q\left( t_{2}\right)
dt_{2}dt_{1}+\int_{0}^{\tau }\int_{t_{1}}^{\tau }q\left( t_{2}\right)
dt_{2}dt_{1}  \nonumber \\
A_{2} &=&\int_{0}^{\tau }\int_{0}^{t_{1}}\int_{t_{1}}^{\tau
}\int_{t_{1}}^{t_{3}}q\left( t_{2}\right) q\left( t_{4}\right)
dt_{4}dt_{3}dt_{2}dt_{1}+\int_{0}^{\tau }\int_{t_{1}}^{\tau
}\int_{0}^{t_{1}}\int_{t_{3}}^{t_{1}}q\left( t_{2}\right) q\left(
t_{4}\right) dt_{4}dt_{3}dt_{2}dt_{1}  \nonumber 
\end{eqnarray}%
\begin{eqnarray}
A_{3} &=&\int_{0}^{\tau }\int_{0}^{t_{1}}\int_{t_{1}}^{\tau
}\int_{t_{1}}^{t_{3}}\int_{t_{3}}^{\tau }\int_{t_{3}}^{t_{5}}q\left(
t_{2}\right) q\left( t_{4}\right) q\left( t_{6}\right)
dt_{6}dt_{5}dt_{4}dt_{3}dt_{2}dt_{1} + \label{re1} \\
&&\qquad\qquad\qquad\qquad+\int_{0}^{\tau }\int_{t_{1}}^{\tau
}\int_{0}^{t_{1}}\int_{t_{3}}^{t_{1}}\int_{0}^{t_{3}}\int_{t_{5}}^{t_{3}}q%
\left( t_{2}\right) q\left( t_{4}\right) q\left( t_{6}\right)
dt_{6}dt_{5}dt_{4}dt_{3}dt_{2}dt_{1}  \nonumber
\end{eqnarray}%
and the general term 
\begin{eqnarray}
A_{\bar{n}} &=&\int_{0}^{\tau }\int_{0}^{t_{1}}\int_{t_{1}}^{\tau
}\int_{t_{1}}^{t_{3}}\ldots \int_{t_{2\bar{n}-3}}^{\tau }\int_{t_{2\bar{n}%
-3}}^{t_{2\bar{n}-1}}q\left( t_{2}\right) \ldots q\left( t_{2\bar{n}%
-2}\right) q\left( t_{2\bar{n}}\right) dt_{2\bar{n}}dt_{2\bar{n}-1}\ldots
dt_{4}dt_{3}dt_{2}dt_{1}+  \nonumber \\
&&+\int_{0}^{\tau }\int_{t_{1}}^{\tau
}\int_{0}^{t_{1}}\int_{t_{3}}^{t_{1}}\ldots \int_{0}^{t_{2\bar{n}%
-3}}\int_{t_{2\bar{n}-1}}^{t_{2\bar{n}-3}}q\left( t_{2}\right) \ldots
q\left( t_{2\bar{n}-2}\right) q\left( t_{2\bar{n}}\right) dt_{2\bar{n}}dt_{2%
\bar{n}-1}\ldots dt_{4}dt_{3}dt_{2}dt_{1}  \nonumber
\end{eqnarray}%
for $\bar{n}=2,3,4$, then, the sequence may be written as%
\begin{eqnarray*}
A_{1} &=&\tau \int_{0}^{\tau }q\left( t_{2}\right) dt_{2} \\
A_{2} &=&\int_{0}^{\tau }\int_{0}^{t_{2}}\left( \tau -t_{2}+t_{1}\right)
\left( t_{2}-t_{1}\right) p\left( t_{1}\right) p\left( t_{2}\right)
dt_{1}dt_{2} \\
A_{3} &=&\int_{0}^{\tau }\int_{0}^{t_{6}}\int_{0}^{t_{4}}\left( \tau
-t_{6}+t_{2}\right) \left( t_{6}-t_{4}\right) \left( t_{4}-t_{2}\right)
q\left( t_{2}\right) q\left( t_{4}\right) q\left( t_{6}\right)
dt_{2}dt_{4}dt_{6}
\end{eqnarray*}%
and the general term 
\[
A_{n}=\int_{0}^{\tau }dt_{1}\int_{0}^{t_{1}}dt_{2}\ldots \int_{0}^{t_{\bar{n}%
-1}}\left( \tau -t_{1}+t_{n}\right) \left( t_{1}-t_{2}\right) \ldots \left(
t_{n-1}-t_{n}\right) q\left( t_{1}\right) q\left( t_{2}\right) \ldots
q\left( t_{n}\right) dt_{n}
\]
\end{lemma}

\begin{proof}
By rearranging the integration variables of the three first coefficients
given in (\ref{re1}) we get

\begin{eqnarray*}
A_{1} &=&\int_{0}^{\tau }\int_{0}^{t_{1}}q\left( t_{2}\right)
dt_{2}dt_{1}+\int_{0}^{\tau }\int_{t_{1}}^{\tau }q\left( t_{2}\right)
dt_{2}dt_{1} \\
&=&\int_{0}^{\tau }\int_{0}^{\tau }q\left( t_{2}\right) dt_{2}dt_{1} \\
&=&\int_{0}^{\tau }\int_{0}^{\tau }q\left( t_{2}\right) dt_{1}dt_{2} \\
&=&\tau \int_{0}^{\tau }q\left( t_{2}\right) dt_{2}
\end{eqnarray*}

\begin{eqnarray*}
A_{2} &=&\int_{0}^{\tau }\int_{0}^{t_{1}}\int_{t_{1}}^{\tau
}\int_{t_{1}}^{t_{3}}q\left( t_{2}\right) q\left( t_{4}\right)
dt_{4}dt_{3}dt_{2}dt_{1}+\int_{0}^{\tau }\int_{t_{1}}^{\tau
}\int_{0}^{t_{1}}\int_{t_{3}}^{t_{1}}q\left( t_{2}\right) q\left(
t_{4}\right) dt_{4}dt_{3}dt_{2}dt_{1} \\
&=&\int_{0}^{\tau }\int_{0}^{t_{4}}\left( \tau -t_{4}\right) \left(
t_{4}-t_{2}\right) q\left( t_{2}\right) q\left( t_{4}\right)
dt_{2}dt_{4}+\int_{\delta }^{\tau }\int_{0}^{t_{4}}\left( t_{4}-t_{2}\right)
t_{2}q\left( t_{2}\right) q\left( t_{4}\right) dt_{2}dt_{4} \\
&=&\int_{0}^{\tau }\int_{0}^{t_{2}}\left( \tau -t_{2}+t_{1}\right) \left(
t_{2}-t_{1}\right) q\left( t_{1}\right) q\left( t_{2}\right) dt_{1}dt_{2}
\end{eqnarray*}

\begin{eqnarray*}
A_{3} &=&\int_{0}^{\tau }\int_{0}^{t_{1}}\int_{t_{1}}^{\tau
}\int_{t_{1}}^{t_{3}}\int_{t_{3}}^{\tau }\int_{t_{3}}^{t_{5}}q\left(
t_{2}\right) q\left( t_{4}\right) q\left( t_{6}\right)
dt_{6}dt_{5}dt_{4}dt_{3}dt_{2}dt_{1}+ \\
&&\qquad\qquad\qquad\qquad\qquad+\int_{0}^{\tau }\int_{t_{1}}^{\tau
}\int_{0}^{t_{1}}\int_{t_{3}}^{t_{1}}\int_{0}^{t_{3}}\int_{t_{5}}^{t_{3}}q%
\left( t_{2}\right) q\left( t_{4}\right) q\left( t_{6}\right)
dt_{6}dt_{5}dt_{4}dt_{3}dt_{2}dt_{1} \\
&=&\int_{0}^{\tau }\int_{0}^{t_{6}}\int_{0}^{t_{4}}\left( \tau -t_{6}\right)
\left( t_{6}-t_{4}\right) \left( t_{4}-t_{2}\right) q\left( t_{2}\right)
q\left( t_{4}\right) q\left( t_{6}\right) dt_{2}dt_{4}dt_{6}+ \\
&&\qquad\qquad\qquad\qquad\qquad+\int_{0}^{\tau }\int_{0}^{t_{2}}\int_{0}^{t_{4}}\left( t_{4}-t_{6}\right)
\left( t_{2}-t_{4}\right) t_{6}q\left( t_{2}\right) q\left( t_{4}\right)
q\left( t_{6}\right) dt_{6}dt_{4}dt_{2} \\
&=&\int_{0}^{\tau }\int_{0}^{t_{6}}\int_{0}^{t_{4}}\left( \tau
-t_{6}+t_{2}\right) \left( t_{6}-t_{4}\right) \left( t_{4}-t_{2}\right)
q\left( t_{2}\right) q\left( t_{4}\right) q\left( t_{6}\right)
dt_{2}dt_{4}dt_{6}
\end{eqnarray*}%
and so on. The general term is then%
\[
A_{n}=\int_{0}^{\tau }dt_{1}\int_{0}^{t_{1}}dt_{2}\ldots \int_{0}^{t_{\bar{n}%
-1}}\left( \tau -t_{1}+t_{n}\right) \left( t_{1}-t_{2}\right) \ldots \left(
t_{n-1}-t_{n}\right) q\left( t_{1}\right) q\left( t_{2}\right) \ldots
q\left( t_{n}\right) dt_{n} 
\]%
and the lemma follows.
\end{proof}

\section*{References}
\bibliography{A Novel Discriminant Approximation of Periodic Differential Equations}

\end{document}